\documentclass[oneside,english,11pt,final]{article}

\usepackage[utf8]{inputenc}
\usepackage[T1]{fontenc}
\usepackage{amssymb,amsthm,amsmath}
\usepackage{babel}
\usepackage{svn}
\usepackage[notref,notcite]{showkeys}
\usepackage{graphicx}
\usepackage[usenames]{color}
\usepackage[unicode=true,
 bookmarks=false,
 breaklinks=false,pdfborder={0 0 1},backref=false,colorlinks=false]
 {hyperref}

\newcommand{\p}{\mathbb{P}}

\newcommand{\ev}[1]{\mathbb{E}{#1}}
\newcommand{\e}[1]{\mathbb{E}}
\newcommand{\pr}[1]{\mathbb{P}\rbr{#1}}
\newcommand{\norm}[3]{\Vert #1 \Vert_{ #2 } ^{ #3 }}
\newcommand{\rbr}[1]{\left( #1 \right)}
\newcommand{\sbr}[1]{\left[ #1 \right]}
\newcommand{\cbr}[1]{\left\{ #1 \right\}}
\newcommand{\ddp}[2]{\left\langle #1, #2 \right\rangle}
\newcommand{\intr}{\int_{\mathbb{R}^d}}
\newcommand{\inti}{\int_{0}^{+\infty}}
\newcommand{\intc}[1]{\int_{0}^{#1}}
\newcommand{\Rd}{\mathbb{R}^d}

\newcommand{\R}{\mathbb{R}}
\newcommand{\T}[1]{\mathcal{T}_{#1}}

\DeclareMathOperator{\var}{Var}
\DeclareMathOperator{\cov}{Cov}
\DeclareMathOperator{\grad}{grad}

\definecolor{czerwony}{rgb}{1,0.3,0.3}
\definecolor{zielony}{rgb}{0.1,0.8,0.4}
\definecolor{blue}{rgb}{0.1,0.1,0.9}

\newcommand{\dd}[1]{\textnormal{d}#1}

\newcommand{\ceq}{\eqsim}
\newcommand{\cleq}{\lesssim}

\newtheorem{theorem}{Theorem}[section]
\newtheorem{lemma}[theorem]{Lemma}
\newtheorem{fact}[theorem]{Fact}

\newtheorem{cor}[theorem]{Corollary}

\theoremstyle{remark}

\usepackage[left=2cm,top=2cm,right=2cm,,bottom=3cm,nohead]{geometry}

\newcommand{\eq}{\varphi}

\newcommand{\tree}[1]{\mathbb{T}_{#1}}
\DeclareMathOperator{\lab}{lab}
\DeclareMathOperator{\Lip}{Lip}
\newcommand{\cspace}[1]{\mathcal{C}^{#1}_c}
\newcommand{\pspace}[1]{\mathcal{P}^{#1}}

\title{$U$-statistics of Ornstein-Uhlenbeck branching particle system}
\author{\textbf{Radosław Adamczak\footnote{radamcz@mimuw.edu.pl. Research was partially supported by the MNiSW grant N N201 397437 and by the
Foundation for Polish Science.} and Piotr Miłoś\footnote{pmilos@mimuw.edu.pl (corresponding author). Research was partially supported by the MNiSW grant N N201 397537.}}\\Faculty of Mathematics, Informatics and Mechanics, University of Warsaw\\ul. Banacha 2, Warsaw, Poland}
\begin{document}

\maketitle

    \begin{abstract}
       We consider a branching particle system consisting of particles moving according to the Ornstein-Uhlenbeck process in $\Rd$ and undergoing a binary, supercritical branching with a constant rate $\lambda>0$. This system is known to fulfil a law of large numbers (under exponential scaling). Recently the question of the corresponding central limit theorem has been addressed. It turns out that the normalization and form of the limit in the CLT fall into three qualitatively different regimes, depending on the relation between the branching intensity and the parameters of the Orstein-Uhlenbeck process. In the present paper we extend those results to $U$-statistics of the system proving a law of large numbers and a central limit theorem.
~
\end{abstract}

\section{Introduction} 
    \label{sec:introduction}
   We consider a single particle located at time $t=0$ at $x \in \R^d$, moving  according to the Orstein-Uhlenbeck process and branching after exponential time independent of the spatial movement. The branching is supercritical and given by the generating function
    \[
        F(s) := p s^2 + (1-p), \quad p>\frac{1}{2}.
    \]
The offspring particles follow the same dynamics (independently of each other). We will refer to this system of particles as the OU branching process and denote it by $X = \cbr{X_t}_{t\ge 0}$.


Formally, we identify the system with the empirical process, i.e. $X$ takes values in the space of Borel measures on $\R^d$ and for each Borel set $A$, $X_t(A)$ is the (random) number of particles at time $t$ in $A$.

It is well known (see e.g. \cite{Harris:2008wd}) that the system satisfies the law of large numbers, i.e. for all bounded continuous functions, conditionally on the set of non-extinction
\begin{equation}
    |X_t|^{-1} \ddp{X_t}{f} \rightarrow \ddp{\eq}{f},\quad a.s. \label{eq:basicLLN},
\end{equation}
where $|X_t|$ is the number of particles at time $t$, $\cbr{X_t(1), X_t(2),\ldots, X_t(|X_t|)}$ are their positions, $\ddp{X_t}{f} := \sum_{i=1}^{|X_t|}f(X_t(i))$ and $\eq$ is the invariant measure of the Orstein-Uhlenbeck process.


In a recent article \cite{Adamczak:2011kx}, we investigated second order behaviour of this system and proved central limit theorems corresponding to (\ref{eq:basicLLN}). It turns out that the behaviour of the system may fall into three qualitatively different categories, depending on the relation between the branching intensity and the parameters of the Orstein-Uhlenbeck process.

In the present article we extend these results on the LLN and CLT to the case of $U$-statistics of the system of arbitrary order $n \ge 1$, i.e. to random variables of the form
\begin{equation}
    U^n_t(f) := \sum_{ \substack{i_1,i_2, \ldots i_n=1 \\ i_k \neq i_j\text{ if } k\neq j}}^{|X_t|} f(X_t(i_1), X_t(i_2), \ldots, X_t(i_n))\label{eq:ustat}
\end{equation}
(note that with this notation $\ddp{X_t}{f}$ corresponds to $U$-statistics of order $n=1$). Our investigation parallels the classical and  well-developed theory of $U$-statistics of independent random variables, however we would like to point out that in our context additional interest in this type of functionals of the process $X$ stems from the fact that they capture 'average dependencies' between particles of the system. This will be seen from the form of the limit, which turns out to be more complicated than in the i.i.d. case.

The organization of the paper is as follows. After introducing the basic notation and preliminary facts in Section \ref{sec:prelim} we describe the main results of the paper in Section \ref{sec:results}.
Next (Section \ref{sec:k_equal_one}) we restate the results in the special case of $n=1$ (as proven in \cite{Adamczak:2011kx}) to serve as a starting point for the general case. Finally, in Section \ref{sec:proofs} we provide complete proofs for arbitrary $n$. We conclude with some remarks concerning the so called non-degenerate case (Section \ref{sec:last}).

\section{Preliminaries} \label{sec:prelim}
\subsection{Notation}
\label{sec:definitions_and_notation}
For a branching system $\cbr{X_t}_{t\geq 0}$, we denote by $|X_t|$ the number of particles at time $t$, and by $X_t(i)$ - the position of the $i$-th (in a certain ordering) particle at time $t$. We sometimes use $\ev{}_x$ or $\mathbb{P}_x$ to denote the fact that we calculate the expectation for the system starting from a particle located at $x$. We use also $\ev{}$ and $\mathbb{P}$ when this location is not relevant.

By $\rightarrow^d $ we denote the convergence in law. We use $\cleq, \ceq$ to denote the situation when an equality or inequality holds with a constant $c>0$, which is irrelevant for calculations. E.g. $f(x)\ceq g(x)$ means that there exists a constant $c>0$ such that $f(x) = c g(x)$.

By $x \circ y = \sum_{i=1}^d x_i y_i$ we denote  the standard scalar product of $x,y\in \Rd$, by $\|\cdot\|$ the corresponding Euclidean norm. By $\otimes^n$ we denote the $n$-fold tensor product.

We use also $\ddp{f}{\mu}:=\int_{\Rd} f(x) \mu{(\dd{x}) }$. 

In our paper we will use the Feynman diagrams. A diagram $\gamma$ labeled by $\cbr{1,2,\ldots}$ is a graph consisting of a set of edges $E_\gamma$ not having common endpoints, and unpaired vertices $A_\gamma$. We will use $r(\gamma)$ to denote the rank of the diagram i.e. the number of edges. For properties and more information we refer to \cite[Definition 1.35]{Janson:1997fk}.

In the paper we will use the space
\begin{equation}
	\pspace{}=\pspace{}(\R^{d}):=\cbr{f:\R^{d}\mapsto\R:f\,\text{ is continuous and } \exists_{n}\text{ such that }|f(x)|/\norm x{}n\rightarrow 0\text{ as }\norm x{}{}\rightarrow+\infty}, \label{eq:pspace}
\end{equation}
We endow this space with the following norm
\begin{equation}
    \norm{f}{\pspace{}}{} =\norm{f}{\pspace{}(\R^d)}{} :=  \sup_{x\in \R^d} |n(x)f(x)|,
\end{equation}
where $n(x):=\exp\rbr{-\sum_{i=1}^d |x_i|}$. We will also use
\[
	\cspace{} = \cspace{}(\R^d),
\]
to denote the space of continuous compactly supported functions.

Given a function $f\in \pspace{}(\R^d)$ we will implicitly understand its derivatives (e.g. $\frac{
\partial f}{
\partial x_i}$) in the space of tempered distributions (see e.g. \cite[p. 173]{Rudin:1973fk}).
\subsection{Basic facts on the Galton-Watson process}
The number of particles $\cbr{|X_t|}_{t\geq 0}$ is the celebrated Galton-Watson process. We present basic properties of this process used in the paper. The main reference in this section is \cite{Athreya:2004xr}.
In our case the expected total number of particles grows exponentially at the rate
\begin{equation}
    \lambda_p :=(2p-1)\lambda. \label{eq:growth-rate}
\end{equation}
The process becomes extinct with probability (see \cite[Theorem I.5.1]{Athreya:2004xr})
\[
    p_e = \frac{1-p}{p}.
\]
We will denote the extinction and non-extinction events by $Ext$ and $Ext^c$ respectively. The process $V_t := e^{-\lambda_p t} |X_t|$ is a positive martingale. Therefore it converges (see also \cite[Theorem 1.6.1]{Athreya:2004xr})
\begin{equation}
    V_t \rightarrow V_\infty, \quad a.s. \:\:\text{ as }t\rightarrow +\infty. \label{eq:defV}
\end{equation}

We have the following simple fact (we refer to \cite{Adamczak:2011kx} for the proof).

\begin{fact} \label{fact:aaa}
    We have $\cbr{V_\infty = 0} = Ext$ and conditioned on non-extinction $V_\infty$ has the exponential distribution with parameter $\frac{2p-1}{p}$. We have $\ev{}(V_\infty) =1$ and $ \var(V_\infty) = \frac{1}{2p-1}$. $\ev{}e^{-4\lambda_p t} |X_t|^4$ is uniformly bounded, i.e. there exists $C>0$ such that for any $t\geq 0$ we have $\ev{}e^{-4\lambda_p t}|X_t|^4 \leq C$. Moreover, all moments are finite, i.e. for any $n\in \mathbb{N}$ and $t\geq 0$ we have $\ev{} |X_t|^n < +\infty$.
\end{fact}
We will denote the variable $V_\infty$ conditioned on non-extinction by $W$.

\subsection{Basic facts on the Orstein-Uhlenbeck process}
\label{sub:ornstein_uhlenbeck}
We recall that the Ornstein-Uhlenbeck process is a time homogenous Markov process with the infinitesimal operator
    \begin{equation}
        L := \frac{1}{2}  \sigma^2 \Delta - \mu x\circ \nabla. \label{eq:ouOperator}
    \end{equation}
The corresponding semigroup will be denoted by $\T{}$. The density of the invariant measure of the Ornstein-Uhlenbeck process is given by
\begin{equation}
    \eq(x) := \rbr{\frac{\mu}{\pi \sigma^2}}^{d/2}\exp \rbr{-\frac{\mu}{\sigma^2} \norm{x}{}{2}}.   \label{eq:equilibrium}
\end{equation}

\subsection{Basic facts concerning $U$-statistics}

We will now briefly recall basic notation and facts concerning $U$-statistics. A $U$-statistic of degree $n$ based on an $\mathcal{X}$-valued sample $X_1,\ldots,X_N$
and a function $f \colon \mathcal{X}^n \to \R$, is a random variable of the form
\begin{displaymath}
\sum_{\substack{i_1,i_2, \ldots i_n=1 \\ i_k \neq i_j\text{ if } k\neq j}}^N f(X_{i_1},\ldots,X_{i_N}).
\end{displaymath}

The function $f$ is usually referred to as the kernel of the $U$-statistic. Without loss of generality it can be assumed that $f$ is symmetric i.e.  invariant under permutation of its arguments. We refer the reader to \cite{Lee:1990pd,Pena:1999bh} for more information on $U$-statistics of sequences of independent random variables.

In our case we will consider $U$-statistics based on the sequence of positions of particles from the branching system as defined by \eqref{eq:ustat}. We will be interested in weak convergence of properly normalised $U$-statistics when $t \to \infty$. Similarly as in the classical theory, the asymptotic behaviour of $U$-statistics depends heavily on the so called order of degeneracy of the kernel $f$, which we will briefly recall in Section \ref{sec:aux_Ustat}.

A function $f$ is called completely degenerate or canonical (with respect to some measure of reference $\eq$, which in our case will be the stationary measure of the
Ornstein-Uhlenbeck process) if
\begin{displaymath}
\int_{\mathcal{X}} f(x_1,\ldots,x_n) \eq(dx_k) = 0,
\end{displaymath}
for all $x_1,\ldots,x_{k-1},x_{k+1},\ldots,x_n \in \mathcal{X}$. The complete degeneracy may be considered a centredness condition, in the classical theory of $U$-statistics canonical kernels are counterparts of centred random variables from the theory of sums of independent random variables.
Their importance stems from the fact that each $U$-statistic can be decomposed into a sum of canonical $U$-statistics of different degrees, a fact known as the Hoeffding decomposition (see Section \ref{sec:aux_Ustat}).


\section{Main results} 
\label{sec:results}
This section is devoted to the presentation of our results. The proofs are deferred to Section \ref{sec:proofs}.

We start with  the following law of large numbers (throughout the article when dealing with $U$-statistics of order $n$ we will identify $\underbrace{\R^d\times\ldots \times\R^d}_{n}$ with $\R^{nd}$).
\begin{theorem} \label{fact:multiLLD}
    Let $\cbr{X_t}_{t\geq 0}$ be the OU branching system starting from $x\in\Rd$. Let us assume that $f:\R^{nd} \mapsto \R$ is a bounded continuous function. Then, on the set of non-extinction $Ext^c$ there is the convergence
    \begin{equation}
        \lim_{n \rightarrow  +\infty} |X_t|^{-n} U_t^n(f) = \ddp{f}{\eq^{\otimes n}} \; \text{a.s.} \label{eq:multiLLN}
    \end{equation}
Moreover, when $f\in \pspace{}(\R^{nd}) $, then the above convergence holds in probability.
\end{theorem}

Having formulated the law of large numbers let us now pass to the corresponding CLTs. As already mentioned in the introduction, their form depends on the relation between $\lambda_p$ and $\mu$, more specifically we distinguish three cases $\lambda_p < 2\mu$, $\lambda_p = 2\mu$ and $\lambda_p > 2\mu$. We refer the reader to \cite{Adamczak:2011kx} for a detailed discussion of this phenomenon as well as its heuristic explanation and interpretation. Here we only stress that the situation for $\lambda_p > 2\mu$ differs substantially from the remaining two cases, as we obtain convergence in probability and the limit is not Gaussian even for $n=1$ (intuitively, this is caused by large branching intensity which lets local correlations between particles prevail over the ergodic properties of the Orstein-Uhlenbeck process).

\subsection{Slow branching case: $\lambda_p < 2\mu$}\label{sec:slowBranching}
Let $Z$ be a Gaussian stochastic measure on $\R^{d+1}$ with intensity $\mu_1(\dd{t} \dd{x} ):=\rbr{\delta_0(\dd{t}) + 2\lambda p e^{\lambda_p t} \dd{t} } \eq(x) \dd{x} $ defined according to \cite[Definition 7.17]{Janson:1997fk}. We denote the stochastic integral with respect to $Z$ by $I$ and the corresponding multiple stochastic integral by $I_n$ \cite[Section 7.2]{Janson:1997fk}. We assume that $Z$ is defined on the probability space $(\Omega, \mathcal{F}, \mathbb{P})$.
For $f\in \pspace{}(\R^{nd})$ we define (we recall that $\T{}$ is the semigroup of the Orstein-Uhlenbeck process)
\begin{equation}
        H(f)(s_1, x_1, s_2,x_2, \ldots, s_n, x_n) := \rbr{\otimes_{i=1}^n \T{s_i}} f(x_1, x_2, \ldots, x_n), \quad s_i \in \R_+,\: x_i\in\R^d. \label{eq:defH}
\end{equation}
It will be useful to treat this function as a function of $n$ variables of type $z_i:=(s_i, x_i) \in \R_+ \times \Rd$. For a Feynman diagram $\gamma$ labeled by $\cbr{1,2, \ldots,n }$
\begin{equation}
    L(f, \gamma ) :=  I_{ |A_\gamma|}\rbr{ \rbr{\prod_{(j,k)\in E_\gamma} {\int \mu_2({\dd{z_{j,k}}}) }} H(f)( u_1, u_2, \ldots, u_n )  },  \quad \mu_2(\dd{t} \dd{x}  ) := 2\lambda p e^{\lambda_p t} \eq(x) \dd{t} \dd{x}, \label{eq:mu2}
\end{equation}
where $u_i = z_{j,k}$ if $(j,k)\in E_\gamma$ and  $i=j$ or $i=k$ and $u_i = z_i $ if $i \in A_\gamma$. Less formally, for each pair $(j,k)$ we integrate over diagonal of coordinates $j$ and $k$ with respect to $\mu_2$. The function obtained in this way is integrated using the multiple stochastic integral $I_{|A_\gamma|}$. We define
\begin{equation}
    L_1(f) := \sum_{\gamma}  L(f, \gamma), \label{eq:limitVaribale}
\end{equation}
where the sum spans over all Feynman diagrams labeled by $\cbr{1, 2, \ldots, n}$.
\begin{fact} \label{fact:well-posed}
    For any canonical $f \in \pspace{}(\R^{nd})$ we have $\ev{L_1(f)}^2 <+\infty$. Moreover  $L_1$ is a continuous function $L_1:\text{Can} \mapsto L_2(\Omega, \mathcal{F}, \mathbb{P})$, where $\text{Can}=\cbr{f \in \pspace{}(\R^{nd}):f\: \text{is a canonical kernel}}$ and $\text{Can}$ is endowed with the norm $\norm{\cdot}{\pspace{}}{}$.
\end{fact}

We are now ready to formulate our main result for processes with the small branching rate.

\begin{theorem}\label{thm:ustatistics-slow} Let $\cbr{X_t}_{t\geq 0}$ be the OU branching system starting from $x\in\Rd$. Let us assume that $f\in \pspace{}(\R^{nd})$ is a canonical kernel and $\lambda_p<2\mu$. Then conditionally on the set of non-extinction $Ext^c$ there is the convergence
    \begin{equation}
        \rbr{e^{-\lambda_p t}|X_t| , \frac{|X_t| - e^{t \lambda_p} V_\infty}{\sqrt{|X_t|}}, \frac{U^n_t(f)}{|X_t|^{n/2}} } \rightarrow^d \rbr{W, G_1 ,L_1(f)}, \label{eq:result}
    \end{equation}
    where where $G_1\sim \mathcal{N}(0, 1/(2p-1))$ and $W, G_1, L_1(f)$ are independent random variables.
\end{theorem}

\subsection{Critical branching case: $\lambda_p = 2\mu$}
Consider the space $\mathcal{L} := L_2(\R^d,\Phi(dx))$, where
\begin{displaymath}
\Phi(x) := \sum_{l=1}^d \Big|\frac{\partial \varphi(x)}{\partial x_l}\Big|
\end{displaymath}
and a centred Gaussian process $(G_f)_{f \in \mathcal{L}}$ defined on some probability space $(\Omega,\mathcal{F},\p)$, with the covariance structure given by
\begin{align}
\cov(G_{f_1},G_{f_2}) = \frac{\lambda p\sigma^2}{\mu}\sum_{l=1}^d \Big\langle f_1,\frac{\partial \varphi}{\partial x_l}\Big\rangle\Big\langle f_2,\frac{\partial \varphi}{\partial x_l}\Big\rangle.\label{eq:criticalCovariances}
\end{align}

We will identify the process with a map $I\colon \mathcal{L} \to L_2(\Omega,\mathcal{F},\p)$, such that $I(f) = G_f$. One can easily check that $I$ is a bounded linear operator.

To formulate the central limit theorem in this case we will need the following
\begin{fact}
\label{fact:well-posed-critical}
For every $n\ge 1$ there exists a unique bounded linear operator $L_2:\text{Can} \mapsto L_2(\Omega, \mathcal{F}, \mathbb{P})$, where $\text{Can}=\cbr{f \in \pspace{}(\R^{nd}):f\: \text{is a canonical kernel}}$ endowed with the norm $\norm{\cdot}{\pspace{}}{}$, such that for every $f_1,\ldots,f_n \colon \R^d\to \R$ satisfying $\langle f_i,\varphi \rangle = 0$ and $f_i \in \pspace{}(\R^d)$, $i=1,\ldots,n$,
\begin{align}\label{eq:tensorization}
L_2(f_1\otimes\cdots\otimes f_n) = I(f_1)\cdots I(f_n).
\end{align}
\end{fact}

The above lemma will be proved together with the following theorem, which describes the asymptotic behaviour of $U$-statistics in the critical case.

\begin{theorem}\label{thm:ustatistics-critical}
Let $\cbr{X_t}_{t\geq 0}$ be the OU branching system starting from $x\in\Rd$. Let us assume that $f\in \pspace{}(\R^{nd})$ is a canonical kernel and $\lambda_p=2\mu$. Then conditionally on the set of non-extinction $Ext^c$ there is the convergence
    \begin{align*}
        \rbr{e^{-\lambda_p t}|X_t| , \frac{|X_t| - e^{t \lambda_p} V_\infty}{\sqrt{|X_t|}}, \frac{U^n_t(f)}{(t|X_t|)^{n/2}} } \rightarrow^d \rbr{W, G ,L_2(f)},
    \end{align*}
where $G_1\sim \mathcal{N}(0, 1/(2p-1))$ and $W, G, L_2(f)$ are independent random variables.
\end{theorem}

\subsection{Fast branching case: $\lambda_p > 2\mu$\label{sec:fast_br}}
In order to describe the limit we introduce an $\Rd$-valued process
\begin{equation}
    H_t := e^{(-\lambda_p + \mu)t} \sum_{i=1}^{|X_t|} X_t(i),\quad t\geq 0. \label{eq:martingale}
\end{equation}
The following two facts have been proved in \cite{Adamczak:2011kx}.
\begin{fact}\label{fact:martingaleConvergence}
    $H$ is a martingale with respect to the filtration of the OU-branching system starting from $x\in\Rd$. Moreover for $\lambda_p>2\mu$ we have $\sup_t \ev{H_t^2} <+\infty$, therefore there exists $H_\infty := \lim_{t\rightarrow +\infty} H_t$ (a.s. limit) and $H_\infty \in L_2$. When the  OU branching system starts from $0$ then martingales $V_t$ and $H_t$ are orthogonal.
\end{fact}

It is worthwhile to note that the distribution of $H_\infty$ depends on the starting conditions.
\begin{fact} \label{fact:law}
    Let $\cbr{X_t}_{t\geq 0}$ and $\{\tilde{X}_t\}_{t\geq 0}$ be two OU branching processes, the first one starting from $0$ and the second one from $x$. Let us denote the limit of corresponding martingales by $H_\infty,\tilde{H}_\infty$ respectively. Then
    \[
        \tilde{H}_\infty  =^d H_\infty +x V_\infty ,
    \]
    where $V_\infty$ is the limit given by \eqref{eq:defV} for the system $X$.
\end{fact}

$H_\infty$ is $\Rd$-valued, we denote its coordinates by $H^i_\infty$. Let $f\in \pspace{}(\R^{nd})$. We define
\[
    \tilde{L}_3(f) := \sum_{i_1,i_2,\ldots, i_n = 1}^d \ddp{\frac{\partial^n f}{\partial x_{1,i_1}\partial x_{2,i_2},\ldots,\partial x_{n,i_n}}}{\eq^{\otimes n}}H_\infty^{i_1}H^{i_2}_\infty \cdots H^{i_n}_\infty.
\]
where we adopted convention that $x_{j,l}$ is the $l$-th coordinate of the $j$-th variable. By $L_3(f)$ we will denote $\tilde{L}_3(f)$ conditioned on $Exp^c$.
\begin{theorem}\label{thm:ustatistics-fast}
    Let $\cbr{X_t}_{t\geq 0}$ be the OU branching system starting from $x\in\Rd$. Let us assume that $f\in \pspace{}(\R^{nd})$ is a canonical kernel and $\lambda_p>2\mu$. Then conditionally on the set of non-extinction $Ext^c$ there is the convergence
        \begin{equation}
            \rbr{e^{-\lambda_p t}|X_t| , \frac{|X_t| - e^{t \lambda_p} V_\infty}{\sqrt{|X_t|}}, e^{-n(\lambda_p-\mu)t } U^n_t(f) } \rightarrow^d \rbr{W, G_1 ,L_3(f)}, \label{eq:result2}
        \end{equation}
        where $G_1\sim \mathcal{N}(0, 1/(2p-1))$ and $(W,L_3(f)), G_1$ are independent. Moreover
        \begin{equation*}
            \rbr{e^{-\lambda_p t}|X_t|, e^{-n(\lambda_p-\mu)t } U^n_t(f) } \rightarrow^d \rbr{V_\infty, \tilde{L}_3(f)},\quad \text{in probability}.
        \end{equation*}
\end{theorem}

\subsection{Remarks on the CLT for $U$-statistics of i.i.d. random variables}

For comparison purposes we will now briefly recall known results on the central limit theorem for $U$-statistics of independent random variables. $U$-statistics were introduced in the 1940's in the context of unbiased estimation by Halmos \cite{Halmos:1946zr} and Hoeffding who obtained the central limit theorem for non-degenerate (degenerate of order 0) kernels \cite{Hoeffding:1948mz}. The full description of the central limit theorem was obtained in \cite{Rubin:1980gf,Dynkin:1983ve} (see also the article \cite{Filippova:1962ly} where the CLT is proven for a related class of $V$-statistics). Similarly as in our case, the asymptotic behaviour of $U$-statistics based on a function $f \colon \mathcal{X}^n\to \R$ and an i.i.d. $\mathcal{X}$-valued sequence $X_1,X_2,\ldots$ is governed by the order of degeneracy of the function $f$ (see Section \ref{sec:aux_Ustat}) with respect to the law of $X_1$ (call it $P$). The case of general $f$ can be reduced to the canonical one, for which one has the weak convergence

\begin{displaymath}
N^{-n/2}\sum_{\substack{i_1,i_2, \ldots i_n=1 \\ i_k \neq i_j\text{ if } k\neq j}}^N f(X_{i_1},\ldots,X_{i_N}) \to J_n(f),
\end{displaymath}
where $J_n$ is the $n$-fold stochastic integral with respect to the so-called isonormal process on $\mathcal{X}$, i.e. the stochastic Gaussian  measure with intensity $P$.

Let us note that in the i.i.d. case the limiting distribution is simpler than for $U$-statistics of the OU branching processes. For small branching rate however, the behaviour of $U$-statistics in our case still resembles the classical one as it is a sum of multiple stochastic integrals of different orders. In the remaining two cases the behaviour differs substantially. This can be seen as a result of the lack of independence. Although asymptotically the particles' positions become less and less dependent, in short time scale offspring of the same particle stay close one to another.

Let us finally mention some results for $U$-statistics in dependent situations, which have been obtained in the last years. In \cite{Borovkova:2001lq} the authors analysed the behaviour of $U$-statistics of stationary absolutely regular sequences and obtained the CLT in the non-degenerate case (with Gaussian limit). In \cite{Borisov:2008dq} the authors considered $\alpha$ and $\varphi$ mixing sequences and obtained a general CLT for canonical kernels. Interesting results for long-range dependent sequences have been also obtained in \cite{Dehling:1987ul}. A more recent interesting work is \cite{Moral:2010uq}, where the authors consider $U$-statistics of interacting particle systems.

\section{The case of $n=1$}
\label{sec:k_equal_one}

In the special case of $n=1$ the results presented in the previous section were proven in \cite{Adamczak:2011kx}.
Although this case obviously follows immediately form the results for general $n$ it is actually a starting point in the proof of the general result (similarly as in the case of $U$-statistics of i.i.d. random variables). Therefore, for the reader's convenience, we will now restate this case in a simpler language of \cite{Adamczak:2011kx}, not involving multiple stochastic integrals.

We will start with the law of large numbers

\begin{theorem}\label{thm:LLN}
    Let $\cbr{X_t}_{t\geq 0}$ be the OU branching system starting from $x\in\Rd$. Let us assume that $f \in \pspace{}(\Rd)$. Then
    \[
        \lim_{t \rightarrow  +\infty} e^{-\lambda_p t}\ddp{X_t}{f} = \ddp{f}{\eq} V_\infty \:\: \text{ in probability},
    \]
    or equivalently on the set of non-extinction, $Ext^c$, we have
    \begin{equation}
        \lim_{t \rightarrow  +\infty} |X_t|^{-1} \ddp{X_t}{f} = \ddp{f}{\eq} \:\: \text{ in probability}. \label{eq:LLN-simple}
    \end{equation}
    Moreover, if $f$ is bounded then the almost sure convergence holds.
\end{theorem}

\subsection{Small branching rate: $\lambda_p < 2\mu$}
We denote $\tilde{f}(x):=f(x)- \ddp{f}{\varphi}$ and
\begin{equation}
    \sigma_f^2 := \ddp{\eq}{\tilde{f}^2} + 2\lambda p  \inti  \ddp{\eq}{\rbr{ e^{(\lambda_p/2) s}\T{s} \tilde{f}}^2}  \dd{s}. \label{eq:sigmaf}
\end{equation}
Let us also recall \eqref{eq:defV} and that $W$ is $V_\infty$ conditioned on $Ext^c$. In this case, the behaviour of $X$ is given by the following
\begin{theorem}\label{thm:clt1}
    Let $\cbr{X_t}_{t\geq 0}$ be the OU branching system starting from $x\in \Rd$. Let us assume that $\lambda_p<2\mu$ and $f\in \pspace{}(\Rd)$. Then $\sigma_f^2 <+\infty$ and conditionally on the set of non-extinction $Ext^c$ there is the  convergence
    \[
        \rbr{e^{-\lambda_p t}|X_t| , \frac{|X_t| - e^{t \lambda_p} V_\infty}{\sqrt{|X_t|}}, \frac{\ddp{X_t}{f} - |X_t|\ddp{f}{\eq} }{ \sqrt{|X_t|} }  } \rightarrow^d (W, G_1, G_2),
    \]
where $G_1\sim \mathcal{N}(0, 1/(2p-1)), G_2\sim \mathcal{N}(0,\sigma_f^2)$ and $W, G_1, G_2$ are independent random variables.
\end{theorem}

\subsection{Critical branching rate: $\lambda_p = 2\mu$}
We denote
\begin{equation}
    \sigma_f^2 := \frac{\lambda p\sigma^2}{\mu} \sum_{i=1}^{d} \ddp{f}{\frac{\partial \eq}{\partial x_i}}^2.
    \label{eq:sigmafCritical}
\end{equation}
Note that the same symbol $\sigma_f^2$ has already been used to denote the asymptotic variance in the small branching case. However, since these cases will always be treated separately, this should not lead to ambiguity.

\begin{theorem}\label{thm:cltCritical}
    Let $\cbr{X_t}_{t\geq 0}$ be the OU branching system starting from $x\in \Rd$. Let us assume that $\lambda_p=2\mu$ and $f\in \pspace{}(\Rd)$. Then $\sigma_f^2 <+\infty$ and conditionally on the set of non-extinction $Ext^c$ there is the  convergence
    \[
        \rbr{e^{-\lambda_p t}|X_t| , \frac{|X_t| - e^{t \lambda_p} V_\infty}{\sqrt{|X_t|}}, \frac{\ddp{X_t}{f} - |X_t|\ddp{f}{\eq} }{ t^{1/2} \sqrt{|X_t|} }  } \rightarrow^d (W, G_1, G_2),
    \]
where $G_1\sim \mathcal{N}(0, 1/(2p-1)), G_2\sim \mathcal{N}(0,\sigma_f^2)$ and $W, G_1, G_2$ are independent random variables.
\end{theorem}

\subsection{Fast branching rate: $\lambda_p > 2\mu$}

In the following theorem we use the notation introduced in Section \ref{sec:fast_br}.

\begin{theorem}\label{thm:clt2}
    Let $\cbr{X_t}_{t\geq 0}$ be the OU branching system starting from $x\in \Rd$. Let us assume that $\lambda_p>2\mu$ and $f\in \pspace{}(\Rd)$. Then conditionally on the set of non-extinction $Ext^c$ there is the  convergence
    \begin{equation}
        \rbr{e^{-\lambda_p t}|X_t| , \frac{|X_t| - e^{t \lambda_p} V_\infty}{\sqrt{|X_t|}}, \frac{\ddp{X_t}{f} - |X_t|\ddp{f}{\eq} }{ \exp\rbr{(\lambda-\mu)t} }  } \rightarrow^d (W, G, \ddp{\grad f}{\eq} \circ J), \label{eq:triple}
    \end{equation}
where $G\sim \mathcal{N}(0, 1/(2p-1))$, $(W,J), G$ are independent and $J$ is $H_\infty$ conditioned on $Ext^c$. Moreover
\[
    \rbr{e^{-\lambda_p t}|X_t|, \frac{\ddp{X_t}{f} - |X_t|\ddp{f}{\eq} }{ \exp\rbr{(\lambda-\mu)t} }  } \rightarrow (V_\infty,\ddp{\grad f}{\eq} \circ H_\infty), \quad \text{in probability.}
\]
\end{theorem}

\section{Proofs} 
\label{sec:proofs}
We will now pass to the proofs of the results announced in Section \ref{sec:results}. Their general structure is similar as in the case of $U$-statistics of independent random variables, i.e. all the theorems will be proved first for linear combinations of tensor products and then via suitable approximations extended to proper function spaces.

In the next section we will recall some additional (standard) facts concerning the Orstein-Uhlenbeck process and $U$-statitistics. Next,
in Section \ref{sec:book_keeping} we will develop general tools needed for the approximation, which will be the most technical part of the proof. Section \ref{sub:proof_of_} will be devoted to rather short proofs of the main results.

From now on we will often work conditionally on the set of non-extinction $Ext^c$, which  will not be explicitly mentioned in the proofs (however should be clear from the context).

\subsection{Auxiliary facts and notation}

\subsubsection{The Orstein-Uhlenbeck process}

The semigroup of the Ornstein-Uhlenbeck process can be represented by
\begin{equation}
    \T{t} f(x) = (g_t \ast f)(x_t), \quad x_t:= e^{-\mu t} x, \label{eq:OUrep}
\end{equation}
where
\[
    g_t(x) = \rbr{\frac{\mu}{\pi \sigma_t^2}}^{d/2}\exp \cbr{-\frac{\mu}{\sigma_t^2} \norm{x}{}{2}}, \quad \sigma_t^2 := \sigma^2(1-e^{-2\mu t}).
\]
Let us recall \eqref{eq:equilibrium}. We denote $ou(t) := \sqrt{1- e^{-2\mu t}}$ and let $G\sim \eq$. The semigroup $\T{}$ has the following useful representations
\begin{equation}
    \T{t} f(x) = \intr f(x_t - y)g_t(y) \dd{y} =  \intr f\rbr{x e^{-\mu t} +  ou(t) y } \eq(y) \dd{y} = \ev{} f(xe^{-\mu t} + ou(t) G). \label{eq:semi-group}
\end{equation}
We also denote
\[
    \T{s}^{\lambda} := e^{\lambda s} \T{s}.
\]

\subsubsection{$U$- and $V$-statistics \label{sec:aux_Ustat}}

We will now briefly recall one of the standard tools of the theory of $U$-statistics, which we will use in the sequel, namely the Hoeffding decomposition.

Let us introduce for $I \subseteq \{1,\ldots,n\}$ the Hoeffding projection of $f$ corresponding to $I$ as the function $\Pi_I f \colon \R^{|I|d} \to \R$, given by
the formula
\begin{equation}
    \Pi_I f(\rbr{x_i}_{i\in I}) = \int_{\R^{nd}} \Big(\prod_{i\notin I} \eq(d y_i) \prod_{i \in I} (\delta_{x_i} - \eq)(dy_i)\Big) f(y_1,\ldots,y_n). \label{eq:hoefdingprojection}
\end{equation}
Once can easily see that for $|I| \ge 1$,  $\Pi_I f$ is a canonical kernel. Moreover $\Pi_\emptyset f = \int_{\R^{nd}} f(x_1,\ldots,x_n)\prod_{i=1}^n\eq(dx_i)$.

Note that if $f$ is symmetric (i.e. invariant with respect to permutations of arguments), $\Pi_I f$ depends only on the cardinality of $f$. In this case we speak about the $k$-th Hoeffding projection
($k = 0,\ldots,n$), given by
\begin{displaymath}
\Pi_k f(x_1,\ldots,x_k) = \ddp{(\delta_{x_1} - \eq)\otimes\cdots\otimes(\delta_{x_k} - \eq)\otimes \eq^{\otimes(n-k)}}{f}.
\end{displaymath}

A symmetric kernel in $n$ variables is called \emph{degenerate of order $k-1$} ($1 \le k \le n$) iff $k = \min\{i> 0 \colon \Pi_i f \not\equiv  0 \}$.
The order of degeneracy is responsible for the normalisation and the form of the limit in the central limit theorem for $U$-statistics, e.g. if the kernel is non-degenerate, i.e. $\Pi_1 f \not\equiv 0$, then the corresponding $U$-statistic of an i.i.d. sequence behaves like a sum of independent random variables and converges to a Gaussian limit. The same phenomenon will be present also in our situation (see Section \ref{sec:last}).

In the particular case $k = n$ the definition of the Hoeffding projection reads as

\[
    \Pi_n(f) := \ddp{\rbr{\otimes_{i=1}^n  (\delta_{x_i} - \eq)}}{f}.
\]
One easily checks that
\begin{displaymath}
f(x_1,\ldots,x_n) = \sum_{I \subseteq \{1,\ldots,n\}} \Pi_If((x_i)_{i\in I}),
\end{displaymath}
which gives us the aforementioned Hoeffding decomposition of $U$-statistics
\begin{displaymath}
U_t^n(f) = \sum_{I\subseteq\{1,\ldots,n\}}\frac{(|X_t|-|I|)!}{(|X_t| - n)!} U_t^{|I|}(\Pi_I f),
\end{displaymath}
which in the case of symmetric kernels simplifies to
\begin{displaymath}
U_t^n(f) = \sum_{k=0}^n\binom{n}{k}\frac{(|X_t|-k)!}{(|X_t| - n)!}  U_t^{k}(\Pi_k f),
\end{displaymath}
where we use the convention $U_t^0(a) = a$ for any constant $a$.

For technical reasons we will also consider the notion of a $V$-statistic which is closely related to $U$-statistics, and is defined as
\begin{equation}
        V^n_t(f) := \sum_{ i_1, i_2, \ldots, i_n = 1 }^{|X_t|} f(X_t(i_1), X_t(i_2), \ldots, X_t(i_n)).\label{eq:vstat}
\end{equation}
The corresponding Hoeffding decomposition is
\begin{equation}
        V^n_t(f) = \sum_{I \subseteq \{1,\ldots,n\}} |X_t|^{n-|I|}V_t^{|I|}(\Pi_I f), \label{eq:tmp666}
\end{equation}
where again we set $V_t^0(a) = a$ for any constant $a$.

In the proof of our results we will use a standard observation that a $U$-statistic can be written as a sum of $V$-statistics. More precisely,
let $\mathcal{J}$ be the collection of partitions of $\{1,\ldots,n\}$ i.e. of all sets $J = \{J_1,\ldots,J_k\}$, where $J_i$'s are nonempty, pairwise disjoint and $\bigcup_i J_i = \{1,\ldots,n\}$. For $J$ as above let $f_J$ be a function of $|J|$ variables $x_1,\ldots,x_{|J|}$, obtained by substituting $x_i$ for all the arguments of $f$ corresponding to the set $J_i$, e.g. for $n = 3$ and $J = \{\{1,2\},\{3\}\}$, $f_J(x_1,x_2) = f(x_1,x_1,x_2)$.
An easy application of the inclusion-exlusion formula yields that
\begin{align}\label{eq:inex}
                U_t^n(f) =  \sum_{J \in \mathcal{J}} a_J V_t^{|J|}(f_J),
\end{align}
where $a_J$ are some integers depending only on the partition $J$. Moreover one can easily check that
if $J = \{ \cbr{1},\ldots,\cbr{n}\}$, then $a_J = 1$, whereas if $J$ consists of sets with at most two elements then $a_J = (-1)^{k}$ where $k$ is the number of two-element sets in $J$. Let us also note that partitions consisting only of one- and two-element sets can be in a natural way identified with the Feynman diagrams (defined in Section \ref{sec:definitions_and_notation}).

\subsection{Approximation } 
\label{sec:book_keeping}

\subsubsection{Approximation of functions}
First we will show that any function in $\pspace{}(\R^{nd})$ can be approximated by tensor functions. For a subset $A$ of a linear space by $span(A)$ we denote the set of finite linear combinations of elements of $A$.

\begin{lemma} \label{lem:approx2}
	Let $A := \cbr{\otimes_{i=1}^n f_i : f_i \; \textrm{bounded continuous} }$ and $f\in \pspace{}(\R^{nd})$ be a canonical kernel then there exists a sequence $\cbr{f_k} \subset span(A)$ such that each $f_k$ is canonical and
    \[
        \norm{f_k-f}{\pspace{}}{} \rightarrow 0, \quad \text{ as } k\rightarrow +\infty.
    \]
	    %
\end{lemma}

\begin{proof} First we prove that $span(A)$ is dense in $\pspace{}(\R^{nd})$. Let us notice that given a function $f\in \pspace{}(\R^{nd})$ it suffices to approximate it on some box $[-M,M]^d$, $M>0$. The box is a compact set and an approximation exists due to the Stone-Weierstrass theorem.
	
	Now, let $f\in \pspace{}(\R^{nd})$. We may find a sequence $\cbr{h_k} \subset span(A)$ such that $h_k \rightarrow f$ in $\pspace{}$. Let us recall the Hoeffding projection \eqref{eq:hoefdingprojection} and denote $I=\cbr{1,2,\ldots,n}$. Now direct calculation (using the exponential integrability of Gaussian variables) reveals that the sequence $f_k := \Pi_{I} h_k$ fulfils the conditions of the lemma.
\end{proof}

Let $f\in \pspace{}(\R^{nd})$ and $I \subset \cbr{1,2,\ldots, n}$ with $I$ = k. We define
\begin{equation}
	\hat{f}_{I} (x_1, x_2, \ldots, x_n) := \int_{\R^{kd}} f(z_1, z_2, \ldots, z_n)  \prod_{i\in I} g_1(x_{i} e^{-\mu}-y_{i})  \dd{y}_{i}, \label{eq:smoothing}
\end{equation}
where $z_i = y_i$ if $i\in I$, $z_i = x_i$ otherwise and $g_1$ is given by \eqref{eq:OUrep}. We have
\begin{lemma}\label{lem:approximationQuality}
	Let $f \in \pspace{}(\R^{nd})$ and $l\in \mathbb{N}$. Then for any $I \subset \cbr{1,2,\ldots, n}$ the function $\hat{f}_I$ is smooth with respect to coordinates in $I$.  For any multi-index $\Lambda = (i_1,\ldots,i_l) \subset \{1,\ldots,nd\}^l$ such that  $\{\lceil i_j /d\rceil \colon j = 1,\ldots,l\} \subset I$ we have
	\[
	    \norm{  \frac{\partial^{|\Lambda|} \hat{f}_{I}}{\partial x_{\Lambda}}  }{\pspace{}}{} \leq C \norm{f}{\pspace{} }{},
	\]
	where $C>0$ depends only on $\sigma,\mu, d,l,n$. Moreover, when $f$ is canonical, so is $\hat{f}_{I}$.
\end{lemma}

\begin{proof}
	Let us fix some $I$ and $\Lambda$. Using \eqref{eq:smoothing} we get
\[
	 K:=   \frac{\partial^{|\Lambda|} \hat{f}_{I}}{\partial x_{\Lambda}}   =  \int_{\R^{kd}} f(z_1, z_2, \ldots, z_n)  \frac{\partial^{|\Lambda|} }{\partial x_{\Lambda}} \prod_{i\in I} g_1(x_{i} e^{-\mu}-y_{i})  \dd{y}_{i}.
\]
Therefore by the properties of Gaussian density $g_1$ and easy calculations we arrive at
\[
	|K| \leq \norm{f}{\pspace{}}{ } \int_{\R^{kd}} n(z_1, z_2, \ldots, z_n)^{-1}  \left| \frac{\partial^{|\Lambda|} }{\partial x_{\Lambda}} \prod_{i\in I} g_1(x_{i} e^{-\mu}-y_{i})\right|  \dd{y}_{i} \leq C_{\Lambda} n(x_1, x_2, \ldots, x_n)^{-1},
\]
for some constant $C_{\Lambda}$.

To conclude it is enough to take the maximum over all admissible pairs $I,\Lambda$.

Let us now assume that $f$ is canonical. We would like to check that for any $j\in \cbr{1,2,\ldots, n}$ we have $$\int_{\R^{nd}} \hat{f}(x_1, x_2, \ldots, x_n) \varphi(x_j) \dd{x}_j=0.$$ There are two cases, the first when $j\notin I$. Then we have
\begin{multline*}
	\int_{\Rd} \hat{f}(x_1, x_2, \ldots, x_n) \varphi(x_j) \dd{x}_j = \int_{\R^d} \int_{\R^{kd}} f(z_1, z_2, \ldots, z_n)  \prod_{i\in I} g_1(x_{i} e^{-\mu}-y_{i})  \dd{y}_{i} \eq(z_j) \dd{z}_j \\= \int_{\R^{kd}} \rbr{\int_{\R^d}f(z_1, z_2, \ldots, z_n) \eq(z_j) \dd{z}_j}  \prod_{i\in I} g_1(x_{i} e^{-\mu}-y_{i})  \dd{y}_{i} =0.
\end{multline*}
The second case is when $j \in I$. Then
\begin{multline*}
	\int_{\Rd} \hat{f}(x_1, x_2, \ldots, x_n) \varphi(x_j) \dd{x}_j \\= \int_{\R^{kd}} f(z_1, z_2, \ldots, z_n)  \prod_{i\in I\setminus \cbr{j}}^k g_1(x_{i} e^{-\mu}-y_{i})  \dd{y}_{i} \rbr{\int_{\R^d} g_1(x_{j} e^{-\mu}-y_{i}) \eq(x_j)\dd{x}_j} \dd{y}_j  \\ = \int_{\R^{kd}} f(z_1, z_2, \ldots, z_n)  \prod_{i\in I\setminus \cbr{j}}^k g_1(x_{i} e^{-\mu}-y_{i})  \dd{y}_{i}  \eq(y_j) \dd{y}_j,
\end{multline*}
where the second equality holds by the fact that $\eq$ is the invariant measure of the Ornstein-Uhlenbeck process. Now the proof reduces to the first case.
\end{proof}

We will also need the following simple identity. We consider $\cbr{x_i}_{i=1,2,\ldots,n},\cbr{\tilde{x}_i}_{i=1,2,\ldots,n}$. By induction one easily checks (we slightly abuse the notation here, e.g. $\frac{\partial }{\partial y_i}$ denotes the derivative in direction $\tilde{x}_i - x_i$ and $\int_a^b$ the integral over the segment $[a,b] \subset \R^d$)
\begin{lemma}\label{lem:derivatives}
Let $f$ be a smooth function, then
    \begin{multline*}
        \sum_{(\epsilon_1,\epsilon_2, \ldots, \epsilon_n) \in \cbr{0,1}^n} (-1)^{\sum_{i=1}^n \epsilon_i}  f(\tilde{x}_1 + \epsilon_1 ({x}_1-\tilde{x}_1), \tilde{x}_2 + \epsilon_2 (x_2-\tilde{x}_2), \ldots, \tilde{x}_n + \epsilon_n (x_n-\tilde{x}_n) ) =\\
        \int_{x_1}^{\tilde{x}_1}\int_{x_2}^{\tilde{x}_2} \ldots \int_{x_n}^{\tilde{x}_n} \frac{\partial^n}{\partial y_1 \partial y_2 \ldots \partial y_n} f(y_1, y_2, \ldots, y_n) \dd{y_n} \dd{y_{n-1}} \ldots \dd{y_1}.
    \end{multline*}
\end{lemma}

\subsubsection{Approximation of $U$-statistics}

\paragraph{Bookkeeping of trees} So far we have shown that one can approximate functions in $\pspace{}$ by linear combinations of tensor products. Our next goal is to show that two functions which are close in $\pspace{}$ generate $U$-statistics which (after proper normalization, specific for each regime) are close in distribution. To this end we will use the ``bookkeeping of trees'' technique (see e.g. \cite{Birkner:2007aa} or \cite[Section 2]{Dynkin:1988uq}), which via some combinatorics and introduction of auxiliary branching processes will allow us to pass from equations on the Laplace transform in the case of $n=1$ to estimates of moments of $V$-statistics and consequently $U$-statistics.

We recall \eqref{eq:vstat}. Let $f_1,f_2, \ldots, f_n \in \cspace{}(\R^d)$ and $f_i \geq 0$. We would like to calculate
\begin{equation}
	 \ev{}_x V^n_t(\otimes_{i=1}^n f_i) = \ev{}_x \rbr{\prod_{i=1}^n \ddp{X_t}{f_i}}. \label{eq:vstatisticMoments}
\end{equation}
Let $\Lambda \subset \cbr{1,2,\ldots, n}$, slightly abusing notation we denote
\[
    w_\Lambda(x,t,\alpha) := \ev{}_x\rbr{\exp \cbr{ - \sum_{i\in \Lambda}  \alpha_i \ddp{X_t}{f_i} }}, \quad w_\Lambda(x,t) := \left.  \frac{\partial^{|\Lambda|} }{\partial \alpha_\Lambda} w(x,t,\alpha)\right|_{\alpha =0}.
\]
Note that this differentiation is valid by Fact \ref{fact:aaa} and properties of the Laplace transform (e.g. \cite[Chapter XIII.2]{Feller:1971cr}). By the calculations from Section 4.2. in \cite{Adamczak:2011kx} we know that
\[
     w_\Lambda(x,t,\alpha) = \T{t} e^{- \sum_{i\in \Lambda}  \alpha_i f_i }(x) + \lambda \intc{t} \T{t-s} \sbr{p w^2_\Lambda(\cdot, s, \alpha) - w_\Lambda(\cdot, s, \alpha) + (1-p)} \dd{s}.
\]
It is easy to check that
\begin{equation}
	v_\Lambda(x,t) := \ev{}_x \rbr{\prod_{i\in \Lambda} \ddp{X_t}{f_i}} = (-1)^{|\Lambda|} w_\Lambda(x,t).	\label{eq:tmp945}
\end{equation}
Assume that $|\Lambda|>0$. We denote by $P_1(\Lambda)$ all pairs $(\Lambda_1, \Lambda_2)$ such that $\Lambda_1 \cup \Lambda_2 = \Lambda$ and $\Lambda_1 \cap \Lambda_2 =\emptyset$, and by $P_2(\Lambda)\subset P_1(\Lambda)$ pairs with an additional restriction that $\Lambda_1 \neq \emptyset$ and $ \Lambda_2 \neq \emptyset$.   We easily check that
\begin{multline*}
    \frac{\partial^{|\Lambda|} }{\partial \alpha_\Lambda} w_\Lambda (x,t,\alpha) = (-1)^{|\Lambda|}\T{t} \rbr{\prod_{i\in\Lambda } f_i \:e^{- \sum_{i\in \Lambda}  \alpha_i f_i }}(x) \\+ \lambda \intc{t} \T{t-s} \sbr{p \sum_{ (\Lambda_1, \Lambda_2) \in P_1(\Lambda) } \frac{\partial^{|\Lambda_1|} }{\partial \alpha_{\Lambda_1}} w_\Lambda(\cdot, s, \alpha) \frac{\partial^{|\Lambda_2|} }{\partial \alpha_{\Lambda_2}} w_\Lambda(\cdot, s, \alpha) - \frac{\partial^{|\Lambda|} }{\partial \alpha_\Lambda} w_\Lambda(\cdot, s, \alpha)}(x) \dd{s}.
\end{multline*}
We evaluate it at $\alpha =0$, (let us notice that $\frac{\partial^{|\Lambda_1|} }{\partial \alpha_{\Lambda_1}} w_\Lambda(x, s, 0) = \frac{\partial^{|\Lambda_1|} }{\partial \alpha_{\Lambda_1}} w_{\Lambda_1}(x, s, 0) = w_{\Lambda_1}(x,s)$), multiply both sides by $(-1)^{|\Lambda|}$ and use the definition of $v_\Lambda(x,t)$ to get
\begin{align*}
    v_\Lambda(x,t) &=  \T{t} \rbr{\prod_{i\in\Lambda } f_i }(x) + \lambda \intc{t} \T{t-s} \sbr{p \sum_{(\Lambda_1, \Lambda_2) \in P_1(\Lambda)  }  v_{\Lambda_1}(\cdot, s) v_{\Lambda_2}(\cdot, s) - v_\Lambda(\cdot, s)}(x) \dd{s}\\
    &=  \T{t} \rbr{\prod_{i\in\Lambda } f_i }(x) + \lambda \intc{t} \T{t-s} \sbr{p \sum_{ (\Lambda_1, \Lambda_2) \in P_2(\Lambda)  }  v_{\Lambda_1}(\cdot, s) v_{\Lambda_2}(\cdot, s) +(2p -1) v_\Lambda(\cdot, s)}(x) \dd{s}.
\end{align*}
This can be easily transformed to (recall that $\T{s}^{\lambda_p} f(x) = e^{\lambda_p s} \T{s}f(x)$)
\begin{equation}
        v_\Lambda(x,t) =  \T{t}^{\lambda_p} \rbr{\prod_{i\in\Lambda } f_i }(x) + p\lambda \intc{t}\T{t-s}^{\lambda_p} \sbr{ \sum_{(\Lambda_1, \Lambda_2) \in P_2(\Lambda)  }  v_{\Lambda_1}(\cdot, s) v_{\Lambda_2}(\cdot, s) }(x) \dd{s}.\label{eq:trees}
\end{equation}
The last formula is much easier to handle if written in terms of auxiliary branching processes. Firstly, we introduce the following notation.
For $n \in \mathbb{N} \setminus \cbr{0}$ we denote by $\tree{n}$ the set of rooted trees described below. The root has a single offspring. All inner vertices (we exclude the root and the leaves) have exactly two offspring. For $\tau \in \tree{n}$, by $l(\tau)$ we denote the set of its leaves. Each leaf $l\in l(\tau)$ is assigned a label, denoted by $\lab(l)$, which is a non-empty subset of $\cbr{1,2,\ldots,n}$, the labels fulfil two conditions:
\[
     \bigcup_{l\in l(\tau)} \lab(l) = \cbr{1,2,\ldots,n}, \quad  \forall_{l_1, l_2 \in l(\tau)}  \rbr{l_1\neq l_2 \implies  \lab(l_1)\cap \lab(l_2) = \emptyset}.
\]
In other words, the labels form a partition of $\cbr{1,2,\ldots,n}$. For a given $\tau \in \tree{n}$ let $i(\tau)$ denote the set of inner vertices (we exclude the root and the leaves), clearly $|i(\tau)| = |l(\tau)|-1$ (as usual $|\cdot|$ denotes the cardinality). Let us identify the vertices of $\tau$ with $\cbr{0,1,2,\ldots, |\tau|-1}$ in such way, that for any vertex $i$ its parent, denoted by $p(i)$, is smaller. Obviously, this implies that $0$ is the root and that the inner vertices have numbers in set $\cbr{1, 2, \ldots, }$. We denote also
\begin{equation}
	s(\tau) := \cbr{l\in l(\tau) : |\lab(l)|=1},\quad m(\tau) := \cbr{l\in l(\tau) : |\lab(l)|>1}, \label{eq:singleMultipleLeaves}
\end{equation}
leaves with single and multiple labels respectively.

    Let $\tau \in \tree{n}$, we consider an Ornstein-Uhlenbeck branching walk on $\tau$ as follows (in this part we ignore the labels). Let us fix $t \in \R_+$ and $\cbr{t_i}_{i \in i(\tau)}$. The initial particle is placed at time $0$ at location $x$, it evolves up to the time $t-t_1$ and splits into two offspring, the first one is associated with the left branch of vertex $1$ in tree $\tau$ and the second one with the right branch. Further each of them evolves until time $t-t_i$, where $i$ is the first vertex in the corresponding subtree, when it splits and so on. At time $t$ the particles are stopped and their positions are denoted by $\cbr{Y_i}_{i\in l(\tau)}$ (the number of particles at the end is equal to the number of leaves). The construction makes sense provided that $t_i \le t$ and $t_i \le t_{p(i)}$ for all $i \in i(\tau)$. We define
\begin{equation}
	   OU\rbr{\otimes_{i=1}^n f_i, \tau, t, \cbr{t_i}_{i \in i(\tau)},x} := \ev{} \rbr{ \prod_{a=1}^n f_a(Y_{j(a)})},\label{eq:ouDefinition}
\end{equation}
    where $j(a) = l \in l(\tau)$ is the unique leaf such that $a\in \lab(l)$. We also define
    \begin{equation}
        S(\tau, t,x) := (p\lambda)^{|i(\tau)|} e^{\lambda_p t} \rbr{ \prod_{i \in i(\tau)} \intc{t}  \dd{t_{i}}  }  \rbr{\prod_{i \in i(\tau)}e^{\lambda_p t_i} 1_{t_{i} \leq t_{p(i)}}} OU \rbr{\otimes_{i=1}^n f_i, \tau, t, \cbr{t_i}_{i \in i(\tau)},x}.  \label{eq:one-tree}
    \end{equation}

\begin{fact} \label{fact:trees-decomposition} Let $\Lambda_n = \cbr{1,2,\ldots,n}$. The following identity holds
    \begin{equation}
        v_{\Lambda_n}(x,t) = \sum_{\tau \in \tree{n}} S(\tau, t,x). \label{eq:coreofthelore}
    \end{equation}
\end{fact}
\begin{proof} The claim is a consequence of the identity $v_{\Lambda}(x,t) = \sum_{\tau \in \tree{}(\Lambda)} S(\tau, t, x)$, where $\Lambda \subset \cbr{1,2,\ldots,n}$ and $\tree{}(\Lambda)$ is the set of trees, as $\tree{|\Lambda|}$, with the exception that the labels are in the set $\Lambda$. This identity in turn will follow by induction with respect to the cardinality of $\Lambda$. For $\Lambda =\cbr{i}$ equation \eqref{eq:trees} reads as  $v_{\Lambda}(x,t) = e^{\lambda_p t}\T{t} f_i(x)$. The space $\tree{}(\Lambda)$ contains only one tree, denoted by $\tau_s$, consisting of the root and a single leaf labelled by $i$. We have $i(\tau) = \emptyset$ and obviously $S(\tau_s, t,x) = e^{\lambda_p t}\T{t} f_i(x)$. Let now $|\Lambda| = k >1$. Similarly as before the first term of \eqref{eq:trees} corresponds to $\tau_s$. By induction the second term can be written as
    \[
        e^{\lambda_p t} p\lambda \intc{t}e^{-\lambda_p t_1}\T{t-t_1} \sbr{ \sum_{(\Lambda_1, \Lambda_2) \in P_2(\Lambda)  }  \sum_{\tau_1 \in \tree{}(\Lambda_1)} \sum_{\tau_2 \in \tree{}(\Lambda_2)} S(\tau_1,t_1,\cdot)  S(\tau_2,t_1,\cdot) }(x) \dd{t}_1.
    \]
     We have
    \begin{multline*}
    \sum_{(\Lambda_1, \Lambda_2) \in P_2(\Lambda)  }  \sum_{\tau_1 \in \tree{}(\Lambda_1)} \sum_{\tau_2 \in \tree{}(\Lambda_2)} e^{\lambda_p t} p\lambda \intc{t}\intr e^{-\lambda_p t_1} p(t - t_1, x,y)  \Big[  \\
        (p\lambda)^{|i(\tau_1)|} e^{\lambda_p t_1} \rbr{ \prod_{i \in i(\tau_1)} \intc{t_1}  \dd{t_{i}}  }  \rbr{\prod_{i \in i(\tau_1)}e^{\lambda_p t_i} 1_{t_{i} \leq t_{p(i)}}} OU \rbr{\otimes_{i=1}^n f_i, \tau_1, t_1, \cbr{t_i}_{i \in i(\tau_1)},y}\\
                (p\lambda)^{|i(\tau_2)|} e^{\lambda_p t_1} \rbr{ \prod_{i \in i(\tau_2)} \intc{t_1}  \dd{t_{i}}  }  \rbr{\prod_{i \in i(\tau_2)}e^{\lambda_p t_i} 1_{t_{i} \leq t_{p(i)}}} OU \rbr{\otimes_{i=1}^n f_i, \tau_2, t_1, \cbr{t_i}_{i \in i(\tau_2)},y}\Big] \dd{y} \dd{t}_1,
    \end{multline*}
    where $p$ is the transition density of the Ornstein-Uhlenbeck process. Now we create a new tree $\tau$ by setting $\tau_1$ and $\tau_2$ to be descendants of the vertex  born at time $t-t_1$. We keep labels and split times unchanged and assume that the first particles of $OU$ processes on $\tau_1$ and $\tau_2$ are put at $x$. Thus by the Markov property of the Ornstein-Uhlenbeck process we can identify the branching random walk on $\tau_1$ and $\tau_2$ with the branching random walk on $\tau$. It is also easy to check that the described correspondence is a bijection from set of pairs $(\tau_1, \tau_2)$ (as in the sum above) to ${\tree{n}\setminus \cbr{\tau_s}}$ and therefore the expression above is equal to the sum $\sum_{\tree{n}\setminus \cbr{\tau_s}} S(\tau, t,x)$.
\end{proof}
The calculations will be more tractable when we derive an explicit formula for $\cbr{Y_i}_{i\in l(\tau)}$. Let us recall the notation introduced in \eqref{eq:semi-group} and consider a family of independent random variables $\cbr{G_i}_{i\in \tau}$, such that $G_i \sim \eq$ for $i\neq 0$ and $G_0 \sim \delta_x$. Recall also that $ou(t) = \sqrt{1 - e^{-2\mu t}}$. The following fact
follows easily from the construction of the branching walk on $\tau$ and (\ref{eq:semi-group}).
\begin{fact} \label{fact:haha}
	Let $\cbr{Y_i}_{i \in l(\tau)}$ be positions of particles at time $t$ of the Ornstein-Uhlenbeck process on tree $\tau$ with labels $\cbr{t_i}_{t\in i(\tau)}$. We have
	\[
        \cbr{Y_i}_{i \in l(\tau)} =^d \cbr{Z_i}_{i \in l(\tau)},
    \]
	where for any $i\in l(\tau)$ we put
    \[
        Z_i:= \sum_{l \in P(i)} ou(t_{p(l)} - t_l) G_l e^{-\mu t_l} + ou(t_{p(i)})G_i,
    \]
    where $P(i) := \cbr{\text{predecessors of } i}$, by convention we set $t_0 = t$ and $ou(t_{p(0)} - t_0) = 1$.
\end{fact}

We are now ready to prove an extended version of Fact \ref{fact:trees-decomposition}.

\begin{fact} \label{fact:f1}
    Let $\cbr{X_t}_{t\geq 0}$ be the OU branching system starting from $x\in\Rd$ and $f\in \pspace{}(\R^{nd}) $ then 
    \begin{equation}
        \ev{} \:\sum_{i_1, i_2, \ldots, i_n =1}^{|X_t|}\: f(X_t(i_1), X_t(i_2), \ldots, X_t(i_n)) = \sum_{\tau \in \tree{n}} S(\tau, t,x), \label{eq:moment-formula2}
    \end{equation}
    where in \eqref{eq:one-tree} we extend the definition of $OU$ in \eqref{eq:ouDefinition} by putting
    \[
            OU(f, \tau, t, \cbr{t_i}_{i \in i(\tau)},x) := \ev{} { f(Y_{j(1)},Y_{j(2)}, \ldots, Y_{j(n)})}.
    \]
    Moreover all the quantities above are finite.
\end{fact}

\begin{proof}(Sketch)
	Using Fact \ref{fact:trees-decomposition}, Fact \ref{fact:aaa}, \eqref{eq:vstatisticMoments} and the Lebesgue monotone convergence theorem together with linearity one may prove that \eqref{eq:moment-formula2} is valid for $f \equiv C$, $C>0$. Using standard tricks we may drop the positivity assumption in \eqref{eq:vstatisticMoments} and \eqref{eq:coreofthelore}. Therefore, by the Stone-Weierstrass theorem and the dominated Lebesgue theorem, \eqref{eq:moment-formula2} is valid for any $f\in \cspace{}(\R^{nd})$.
	Let now $f\in \pspace{}(\R^{nd})$, $f\geq 0$. We notice that for any $\tau \in \tree{n}$ the expression $OU(f, \tau, t, \cbr{t_i}_{i \in i(\tau)},x)$ is finite, which follows easily from Fact \ref{fact:haha}. Further, one can find a sequence $\cbr{f_k}$ such that $f_k \in \cspace{}(\Rd)$, $f_k\geq 0$ and $f_k\nearrow f$ (pointwise). Appealing to the monotone Lebesgue theorem yields that \eqref{eq:moment-formula2} still holds (and is finite). To conclude, once more we remove the positivity condition.
\end{proof}
As a simple corollary we obtain

\begin{cor} \label{cor:godsavethequeen}
	Let $\cbr{X_t}_{t\geq 0}$ be the OU branching system, then for any $n\geq 1$ there exists $C_n$ such that
	\[
		\ev{} |X_t|^n \leq C_n e^{n \lambda_p t}.
	\]
\end{cor}
\begin{proof}
	We apply the above fact with $f=1$. Using definition \eqref{eq:one-tree} and the inequality $|i(\tau)| \le n-1$ for $\tau \in \tree{n}$, it is easy to check that for any $t\in \tree{n}$ we have $S(\tau, t, x)\leq C_\tau e^{n \lambda_p t}$, for a certain constant depending only on $\tau$.
\end{proof}

Let us recall notation of \eqref{eq:singleMultipleLeaves}. We have
\begin{fact}\label{fact:f3Better}
	For any $n\in \mathbb{N}$ there exists $C, c>0$, such that for any $\tau \in \tree{n}$ any $\cbr{t_i}_{i\in i(\tau)}$, split times as above, and any canonical $f\in \pspace{}(\R^{nd})$ we have
	\begin{equation}
		OU(f, \tau, t, \cbr{t_i}_{i \in i(\tau)},x) \leq \norm{f}{\pspace{}}{} \:C \exp\cbr{{c \norm{x}{}{}}}  \exp \Big(-\mu \sum_{i\in s(\tau)} t_{p(i)}\Big). \label{eq:multipledecay}
	\end{equation}
\end{fact}
\begin{proof}
Let $k\leq n$. Without loss of generality we may assume that $I:=\cbr{1,2,\ldots, k}$ are single numbers (i.e. $j(i)\in s(\tau)$) and $\cbr{k+1, \ldots, n}$ are multiple ones. Let us also assume for a moment that for $i\in \cbr{1,2,\ldots, k}$ we have $t_{p(j(i))}\geq 1$. Let $Z_i$ and $G_i$ be as in Fact \ref{fact:haha}. We have $\ev{} f(Y_{j(1)}, Y_{j(2)}, \ldots, Y_{j(n)})= \ev{} f(Z_{j(1)}, Z_{j(2)}, \ldots, Z_{j(n)})$. For $i\leq k$ we define
\[
	  \tilde{Z}_i:= \sum_{l \in P(i)} ou(t_{p(l)} - t_l) G_l e^{-\mu (t_l-1)} + ou(t_{p(i)}-1)G_i.
\]
Moreover, let $\hat{f} := \hat{f}_{I}$ be given by \eqref{eq:smoothing}. By the semigroup property of $\T{}$ we have
\[
	\ev{} f(Z_{j(1)}, Z_{j(2)}, \ldots, Z_{j(n)}) = \ev{} \hat{f}(\tilde{Z}_{j(1)}, \ldots, \tilde{Z}_{j(k)}, Z_{j(k+1)} \ldots, Z_{j(n)}) =:A.
\]
By the fact that $\hat{f}$ is smooth with respect to coordinates from $I$ and canonical and Lemma \ref{lem:derivatives} we obtain
        \begin{align*}
          A  &= \ev{} \sum_{(\epsilon_1, \ldots, \epsilon_k) \in \cbr{0,1}^k} (-1)^{\sum_{i=1}^n \epsilon_i} \hat{f} (\tilde{Z}_{j(1)} + \epsilon_1 (G_{j(1)}-\tilde{Z}_{j(1)}), \ldots, {\tilde{Z}}_{j(k)} + \epsilon_k (G_{j(k)} -{\tilde{Z}}_{j(k)}), Z_{j(k+1)},\ldots, Z_{j(n)} )\\
            &= \ev{}  \int_{G_{j(1)}}^{\tilde{Z}_{j(1)}} \ldots \int_{G_{j(k)}}^{\tilde{Z}_{j(k)}} \frac{\partial^k}{\partial y_1  \ldots \partial y_k} \hat{f}(y_1, \ldots, y_k, Z_{j(k+1)},\ldots, Z_{j(n)}) \dd{y_k} \ldots \dd{y_1}.
    \end{align*}

     From now on we restrict to the case $d=1$. The proof for general $d$ proceeds along the same lines but it is notationally more cumbersome. Using Lemma \ref{lem:approximationQuality} and applying the Schwarz inequality multiple times we have
    \begin{multline*}
        |A| \leq C\norm{f}{\pspace{}}{} \ev{} \left| \rbr{\prod_{i=1}^k \int_{G_{j(i)}}^{\tilde{Z}_{j(i)}} \exp\cbr{|y_i|} \dd{y}_i} \prod_{i=k+1}^n \exp\cbr{ |Z_{j(i)}|} \right| \\
        \leq \norm{f}{\pspace{}}{}  \prod_{i=1}^k \norm{\int_{G_{j(i)}}^{\tilde{Z}_{j(i)}} \exp\cbr{|y_i|} \dd{y_i}  }{2^i}{}  \prod_{i=k+1}^n \norm{ \exp\cbr{|Z_{j(i)}|} }{2^i}{}
    \end{multline*}


    Note that by the definition of $\tilde{Z}_i$ we have $\tilde{Z}_i - G_i =  H_i e^{-\mu (t_{p(i)}-1)} + \rbr{ou(t_{p(i) }-1) - 1}  G_i $, where $H_i$ is independent of $G_i$ and $H_i \sim \mathcal{N}(x_i, \sigma_i^2)$ with $\sigma_i \leq \sigma/\sqrt{2\mu}$ and $\norm{x_i}{}{} \leq \norm{x}{}{}$. Thus $\tilde{Z}_i - G_i$ is a~Gaussian variable with the mean bounded by $C\norm{x_i}{}{}e^{-\mu t_{p(i)}}$ and the standard deviation of order $e^{-\mu t_{p(i)}}$. In particular $\|\tilde{Z}_i-G_i\|_k \le C_k\exp(C_k\|x\|-\mu t_{p(i)})$. Since
    \begin{displaymath}
    \int_{G_{j(i)}}^{\tilde{Z}_{j(i)}} \exp\cbr{|y_i|} \dd{y}_i \le \Big(e^{|G_{j(i)}|} + e^{|\tilde{Z}_{j(i)}|}\Big)|\tilde{Z}_i - G_i|
    \end{displaymath}
    the proof can be concluded by yet another application of the Schwarz inequality and standard facts on exponential integrability of Gaussian variables.

Finally, if some $i$'s do not fulfil $t_{p(i)}\geq 1$ we repeat the above proof with $s(\tau)$ replaced by the set $s'$ of indices for which additionally $t_{p(i)}\geq 1$. In this way we obtain \eqref{eq:multipledecay} with $\sum_{i\in s'} t_{p(i)}$. In our setting $$\sum_{i\in s(\tau)} t_{p(i)} - \sum_{i\in s'} t_{p(i)} = \sum_{i \in s(\tau) \setminus s'} t_{p(i)} \leq |s(\tau) \setminus s'|\leq n,$$
hence \eqref{eq:multipledecay} still holds (with a worse constant $C$).
\end{proof}

\paragraph{Bounds for $U$- and $V$-statistics} Using Fact \ref{fact:f3Better} we will now be able to estimate moments of $V$- and $U$-statistics of the branching particle system. The inequalities will be stated separately for each regime discussed in Section \ref{sec:results}.

We will first develop $L_2$ bounds for $U$-statistics with deterministic normalization. Since in the slow and critical branching case the normalization in our theorems is random, related to the size of the process, we will later transform these bounds to $L_0$-bounds for $U$-statistics with a random normalization.

\begin{fact}[Small branching rate] \label{fact:approximation} Let $\cbr{X_t}_{t\geq 0}$ be the OU branching particle system with $\lambda_p<2\mu$. There exist $C,c >0$ such that for any  canonical kernel $f\in \pspace{}(\R^{nd})$ we have
    \[
             \ev{}_x\rbr{ e^{-(n/2)\lambda_p t} V_t^n(f) }^2 \leq C \exp\cbr{c \norm{x}{}{}} \norm{f}{\pspace{} }{2}.
    \]
\end{fact}

\begin{proof}  We need to estimate $\ev{}_x e^{-n\lambda_p t} \sum_{i_1, i_2,\ldots i_{2n}=1}^{|X_t|} f(X_t(i_1), \ldots,X_t(i_n)) f(X_t(i_2), \ldots,X_t(i_{2n}))$. Obviously the function $f \otimes f$ is canonical. Moreover, it is easy to check, that $\norm{f\otimes f}{\pspace{}}{} \le \norm{f}{\pspace{}}{2}$.

  By Fact \ref{fact:f1} it suffices to show that for each $\tau \in \tree{2n}$ there exist $C,c>0$ such that for any $t>0$ we have $e^{-n\lambda_p t}S(\tau, t,x) \leq C\exp\cbr{c\norm{x}{}{}}\norm{f}{\pspace{}}{2}$.

Let us fix $\tau \in \tree{2n}$ and denote by $P_1(\tau)$ and $P_2(\tau)$ the sets of inner vertices of $\tau$ with respectively one and two children in $s(\tau)$. Set also $P_3(\tau) :=  i(\tau)\setminus(P_1(\tau)\cup P_2(\tau))$.

By the definition  of $S(\tau,t,x)$, Fact \ref{fact:f3Better} and the assumption $\lambda_p < 2\mu$, we get
\begin{align*}
&e^{-n\lambda_p t}S(\tau, t,x) \\
\le& C_1 e^{-(n-1)\lambda_p t} \rbr{ \prod_{i \in i(\tau)} \intc{t}  \dd{t_{i}}  }  \rbr{\prod_{i \in i(\tau)}e^{\lambda_p t_i} 1_{t_{i} \leq t_{p(i)}}} OU \rbr{f, \tau, t, \cbr{t_i}_{i \in i(\tau)},x}\\
\le& C_2 \|f\|_{\pspace{}}^2 \exp(c\|x\|) e^{-(n-1)\lambda_p t} \rbr{ \prod_{i \in i(\tau)} \intc{t}  \dd{t_{i}}  }  \rbr{\prod_{i \in i(\tau)}e^{\lambda_p t_i} 1_{t_{i} \leq t_{p(i)}}\prod_{i \in s(\tau)} e^{-\mu t_{p(i)}}}\\
\le&C_2\|f\|_{\pspace{}}^2 \exp(c\|x\|) \exp\Big((-(n-1)+|P_1(\tau)|/2 + |P_3(\tau)|)\lambda_p t\Big) \\
&\times
\Big(\prod_{i \in P_1(\tau)}\int_0^t e^{(\lambda_p - \mu)t_i - (\lambda_p t)/2}dt_i \Big) \Big(\prod_{i \in P_2(\tau)}\int_0^t e^{(\lambda_p - 2\mu) t_i}dt_i \Big) \Big(\prod_{i \in P_3(\tau)}\int_0^t e^{\lambda_p (t_i -t)}dt_i \Big)\\
\le& C_3\|f\|_{\pspace{}}^2 \exp(c\|x\|) \exp\Big((-(n-1)+|P_1(\tau)|/2 + |P_3(\tau)|)\lambda_p t\Big).
\end{align*}
To end the proof it is thus sufficient to show that $|P_1(\tau)| + 2|P_3(\tau)| \le 2n-2$.

Note that $|P_1(\tau)| + 2|P_2(\tau)| = |s(\tau)|$ and $\sum_{i=1}^3|P_i(\tau)| = |i(\tau)| = |l(\tau)|-1$. Thus
$|P_1(\tau)| + 2|P_3(\tau)| =  2|l(\tau)| - 2 - |s(\tau)|=|l(\tau)| + |m(\tau)| - 2 \le 2n -2$ (recall that $m(\tau)$ denotes the set of leaves with multiple labels). This ends the proof.
\end{proof}

We will now pass to an analogous estimate in the critical case.

\begin{fact}[Critical branching rate]
\label{fact:approximation_critical}
Let $\cbr{X_t}_{t\geq 0}$ be the OU branching particle system with $\lambda_p=2\mu$. There exist $C,c >0$ such that for any  canonical kernel $f\in \pspace{}(\R^{nd})$ we have
    \[
             \ev{}_x\rbr{ t^{-n/2}e^{-(n/2)\lambda_p t} V_t^n(f) }^2 \leq C \exp\cbr{c \norm{x}{}{}} \norm{f}{\pspace{} }{2}.
    \]
\end{fact}
\begin{proof}

We will use similar ideas as in the proof of Fact \ref{fact:approximation} as well as the notation introduced therein. Consider any $\tau\in \tree{2n}$. By the definition  of $S(\tau,t,x)$, Fact \ref{fact:f3Better} and the assumption $\lambda_p = 2\mu$, we obtain
\begin{align*}
&t^{-n}e^{-n\lambda_p t}S(\tau, t,x) \\
\le& C_1
\|f\|_{\pspace{}}^2 \exp\cbr{{c \norm{x}{}{}}}
t^{-n}e^{-(n-1)\lambda_p t} \rbr{ \prod_{i \in i(\tau)} \intc{t}  \dd{t_{i}}  }  \rbr{\prod_{i \in i(\tau)}e^{\lambda_p t_i} 1_{t_{i} \leq t_{p(i)}}\prod_{i \in s(\tau)} e^{-\mu t_{p(i)}}}\\
\le& C_2\|f\|_{\pspace{}}^2 \exp(c\|x\|) t^{-n}\exp\Big((-(n-1)+|P_1(\tau)|/2 + |P_3(\tau)|)\lambda_p t\Big) \\
&\times
\Big(\prod_{i \in P_1(\tau)}\int_0^t e^{(\lambda_p - \mu)t_i - (\lambda_p t)/2}dt_i \Big) \Big(\prod_{i \in P_2(\tau)}\int_0^t e^{(\lambda_p - 2\mu) t_i}dt_i \Big) \Big(\prod_{i \in P_3(\tau)}\int_0^t e^{\lambda_p (t_i -t)}dt_i \Big)\\
\le&C_2\|f\|_{\pspace{}}^2 \exp(c\|x\|),
\end{align*}
where we used the fact that $|P_2(\tau)| \le n$ and the estimate $|P_1(\tau)| + 2|P_3(\tau)| \le 2n-2$ obtained in the proof of Fact \ref{fact:approximation}.

\end{proof}

\begin{fact}[Fast branching rate]\label{fact:approximation_supercritical}Let $\cbr{X_t}_{t\geq 0}$ be the OU branching particle system with $\lambda_p>2\mu$. There exist $C,c >0$ such that for any canonical kernel $f\in \pspace{}(\R^{nd})$ we have
    \[
             \ev{}_x\rbr{ e^{-n(\lambda_p - \mu) t} V_t^n(f) }^2 \leq C\exp\cbr{c\norm{x}{}{}}\norm{f}{\pspace{} }{2}.
    \]
\end{fact}




\begin{proof} As in the previous cases, consider any $\tau \in \tree{2n}$. We have
\begin{align*}
&e^{-2n(\lambda_p -\mu)t}S(\tau, t,x) \\
\le& C_1 e^{-2n(\lambda_p -\mu)t +\lambda_p t}\rbr{ \prod_{i \in i(\tau)} \intc{t}  \dd{t_{i}}  }  \rbr{\prod_{i \in i(\tau)}e^{\lambda_p t_i} 1_{t_{i} \leq t_{p(i)}}} OU \rbr{f, \tau, t, \cbr{t_i}_{i \in i(\tau)},x}\\
\le& C_2 \|f\|_{\pspace{}}^2 \exp(c\|x\|) e^{-2n(\lambda_p -\mu)t +\lambda_p t}\rbr{ \prod_{i \in i(\tau)} \intc{t}  \dd{t_{i}}  }  \rbr{\prod_{i \in i(\tau)}e^{\lambda_p t_i} 1_{t_{i} \leq t_{p(i)}}\prod_{i \in s(\tau)} e^{-\mu t_{p(i)}}}\\
\leq&C_2\|f\|_{\pspace{}}^2 \exp(c\|x\|) e^{-2n(\lambda_p -\mu)t +\lambda_p t} \\
&\times
\Big(\prod_{i \in P_1(\tau)}\int_0^t e^{(\lambda_p - \mu)t_i }dt_i \Big) \Big(\prod_{i \in P_2(\tau)}\int_0^t e^{(\lambda_p - 2\mu) t_i}dt_i \Big) \Big(\prod_{i \in P_3(\tau)}\int_0^t e^{\lambda_p t_i}dt_i \Big)\\
\le& C_3\|f\|_{\pspace{}}^2 \exp(c\|x\|) \exp\Big(-2n(\lambda_p -\mu)t +\lambda_p t + |P_1(\tau)|(\lambda_p - \mu)t + |P_2(\tau)|(\lambda_p - 2\mu)t
+|P_3(\tau)|\lambda_pt\Big).
\end{align*}
Thus it is enough to prove that
\begin{align}\label{eq:approx_super}
\lambda_p  + |P_1|(\lambda_p - \mu) + |P_2|(\lambda_p - 2\mu)
+|P_3|\lambda_p \le 2n(\lambda_p -\mu),
\end{align}
where for simplicity we write $P_i$ instead of $P_i(\tau)$ (in the rest of the proof we will use the same convention with other characteristics of $\tau$). Using the equality $|s| = |P_1| + 2|P_2|$, we may rewrite (\ref{eq:approx_super}) as
\begin{displaymath}
\lambda_p + |s|(\lambda_p - \mu) - |P_2|\lambda_p +|P_3|\lambda_p \le 2n(\lambda_p - \mu),
\end{displaymath}
so by the inequalities $2n \ge |s|$ and $\lambda_p - \mu > \lambda_p/2$ it is enough to prove that
\begin{displaymath}
2 + |s| - 2|P_2| + 2|P_3| \le 2n.
\end{displaymath}
But $|P_3| = |i| - |P_2| - |P_1|$ and so
\begin{displaymath}
2 + |s| - 2|P_2| + 2|P_3| = 2 + |s| + 2|i| - 4|P_2| - 2|P_1| = 2 -|s| + 2|i| = - |s| + 2|l| = 2|m| + |s| \le 2n,
\end{displaymath}
which ends the proof.
\end{proof}

\paragraph{Random normalization} Using the facts obtained above we will now prove estimates for $U$-statistics normalized by a proper power of $|X_t|$, which will be relevant in the proofs of Theorems \ref{thm:ustatistics-slow} and \ref{thm:ustatistics-critical}. Since asymptotically $|X_t|\exp(-\lambda_p t)$ behaves (conditionally on $Ext^c$) as an exponential random variable $W$ and $\ev{} W^{-1}$ does not exist, we will have to introduce a truncation, cutting out the set where $|X_t|$ is small.

\begin{cor} \label{cor:helper} Let $\cbr{X_t}_{t\geq 0}$ be the OU branching system with $\lambda_p<2\mu$. There exist constants $C, c>$ such that for any  canonical $f \in \pspace{}(\R^{nd})$ and $r \in (0,1)$ we have
\begin{displaymath}
\ev{}_x\rbr{ 1_{Ext^c} |X_t|^{-n/2}|U_t^n(f)| 1_{\cbr{|X_t| \ge re^{\lambda_p t}}}} \leq C \exp\cbr{c \norm{x}{}{}} r^{-n/2} \|f(2n\cdot)\|_{\pspace{}}.
\end{displaymath}
\end{cor}

\begin{proof}
Let $\mathcal{J}$ be the collection of partitions of $\{1,\ldots,n\}$ i.e. of all sets $J = \{J_1,\ldots,J_k\}$, where $J_i$'s are nonempty, pairwise disjoint and $\bigcup_i J_i = \{1,\ldots,n\}$. Using \eqref{eq:inex} and notation introduced there we have
\begin{align}\label{eq:inex2}
                |X_t|^{-n/2}U_t^n(f) = |X_t|^{-n/2} \sum_{J \in \mathcal{J}} a_J V_t^{|J|}(f_J),
\end{align}
where $a_J$ are some integers depending only on the partition $J$. Since the cardinality of $\mathcal{J}$ depends only on $n$, it is enough to show that for each $J \in \mathcal{J}$ and some constants $C,c>0$ we have
\begin{displaymath}
\ev{}_x \rbr{|X_t|^{-n/2} | V_t^{|J|}(f_J)| 1_{\cbr {|X_t| \ge re^{\lambda_p t}}}} \le C \exp\cbr{c \norm{x}{}{}} r^{-n/2} \|f(2n\cdot)\|_{\pspace{}}.
\end{displaymath}

Let us thus consider $J = \{J_1,\ldots,J_k\}$ and let us assume that among the sets $J_i$ there are exactly $l$ sets of cardinality $1$, say $J_1,\ldots,J_l$.
We would like to use Fact \ref{fact:approximation}. To this end we have to express $V_t^k(f_J)$ as a sum of $V$-statistics with canonical kernels.
This can be easily done by means of Hoeffding's decomposition \eqref{eq:tmp666}. Since $f_J$ is already degenerate with respect to variables $x_1,\ldots,x_l$, we get
\begin{align}\label{eq:partial_Hoeffding}
|X_t|^{-n/2} V_t^k(f_{J}) = \sum_{I \subseteq \{l+1,\ldots,k\}} |X_t|^{|I| - n/2} V_t^{|I^c|}(\Pi_{I^c} f_J),
\end{align}
where $I^c := \cbr{1,\ldots, k} \setminus I$. Let us notice that $n \ge  2k-l$, so $n - k \ge k - l \ge |I|$, which gives
\begin{align}\label{eq:set_bounds}
n - 2|I| \ge k - |I| = |I^c|.
\end{align}
Thus we have
\begin{multline*}
    \ev{}_x \rbr{|X_t|^{ |I| - n/2} |V_t^{|I^c|}(\Pi_{I^c} f_J)|1_{\cbr{|X_t| \ge re^{\lambda_p t}}}} \le r^{-n/2 + |I|} \ev{} \rbr {e^{- \lambda_p t (n - 2|I|)/2}|V_t^{|I^c|}(\Pi_{I^c} f_J)| }\\
    \le r^{-n/2 + |I|} \ev{} \rbr{e^{- \lambda_p t |I^c|/2}|V_t^{|I^c|}(\Pi_{I^c} f_J)|}
    \le C \exp\cbr{c \norm{x}{}{}}r^{-n/2 + |I|}\|\Pi_{I^c} f_J\|_{\pspace{}}, 
\end{multline*}
for $I \neq \{1,\ldots,k\}$, where in the third inequality we used Fact \ref{fact:approximation}. One can check that for any $n\geq 2$ there exists $C>0$ such that for any $I,J\subset \cbr{1,2,\ldots,n}$ we have $\norm{\Pi_I f}{\pspace{}}{} \leq C \norm{f}{\pspace{}}{}$ and  $\norm{f_J}{\pspace{}}{} \leq \norm{f(2n\cdot)}{\pspace{}}{}$. Therefore it remains to bound the contribution from $I = \{1,\ldots,k\}$ (in the case $l=0$). But in this case $I^c = \emptyset$, so $|V_t^{|I^c|}(\Pi_{I^c} f_J)| = |\Pi_{I^c} f_J| =| \langle \eq^{k},f_J\rangle| \le C\|f\|_\pspace{}$ and $\exp(-\lambda_pt(n-2|I|)) \le 1$, which easily gives the desired estimate.

\end{proof}
In an analogous way, replacing Fact \ref{fact:approximation} by Fact \ref{fact:approximation_critical} one proves

\begin{cor} \label{cor:helper_crit} Let $\cbr{X_t}_{t\geq 0}$ be the OU branching system with $\lambda_p=2\mu$. There exist constants $C, c$ such that for any canonical $f \in \pspace{}(\R^{nd})$  and $r \in (0,1)$ we have for $t \ge 1$,
\begin{displaymath}
\ev{}\rbr{ 1_{Ext^c} (t|X_t|)^{-n/2}|U_t^n(f)| 1_{\{|X_t| \ge re^{\lambda_p t}\}}} \leq C \exp\cbr{c \norm{x}{}{}} r^{-n/2} \|f(2n\cdot)\|_{\pspace{}}.
\end{displaymath}
\end{cor}

We also have an analogous statement in the supercritical case.

\begin{cor} \label{cor:helper_supercrit} Let $\cbr{X_t}_{t\geq 0}$ be the OU branching system with $\lambda_p>2\mu$. There exist constants $C, c$ such that for any canonical $f \in \pspace{}(\R^{nd})$ we have
\begin{displaymath}
\ev{}\rbr{  e^{-n(\lambda_p - \mu)t}|U_t^n(f)|} \leq C \exp\cbr{c \norm{x}{}{}} \|f(2n\cdot)\|_{\pspace{}}.
\end{displaymath}
\end{cor}
\begin{proof} Using the notation from the proof of Fact \ref{cor:helper}, we get
\begin{align*}
\ev{} e^{-n(\lambda_p-\mu)t} |V_t^k(f_{J})| &\le \ev{}\sum_{I \subseteq \{l+1,\ldots,k\}} \frac{|X_t|^{|I|}}{e^{\lambda_p |I| t}} e^{-n(\lambda_p-\mu)t + \lambda_p |I| t} |V_t^{|I^c|}(\Pi_{I^c} f_J)|\\
& \le \sum_{I \subseteq \{l+1,\ldots,k\}} \Big\|\frac{|X_t|^{|I|}}{e^{\lambda_p |I| t}} \Big\|_2\|e^{-n(\lambda_p-\mu)t + \lambda_p |I| t} V_t^{|I^c|}(\Pi_{I^c} f_J)\|_2
\end{align*}
Let us note that by $\lambda_p > 2\mu$  and \eqref{eq:set_bounds} we get $n(\lambda_p - \mu) - \lambda_p|I| \ge n(\lambda_p - \mu) - 2 (\lambda_p - \mu)|I| \ge |I^c|(\lambda_p - \mu)$.
Thus the summands on the right hand side above for $I \neq \{1,\ldots,k\}$ can be bounded using Corollary \ref{cor:godsavethequeen} and Fact \ref{fact:approximation_supercritical} by
\begin{displaymath}
C \exp(c\|x\|)\|\Pi_{I^c} f_J\|_\pspace{} \le C^2\exp(c\|x\|)\|f(2n\cdot)\|
\end{displaymath}
(the last inequality is analogous as in the proof of Fact \ref{cor:helper}).

The contribution from $I = \{1,\ldots,k\}$   (in the case $l=0$) also can be bounded like in Fact \ref{cor:helper}. Namely, $I^c = \emptyset$, so $|V_t^{|I^c|}(\Pi_{I^c} f_J)| = |\Pi_{I^c} f_J| =| \langle \eq^{k},f_J\rangle| \le C\|f\|_\pspace{}$ and $\exp(-n(\lambda_p-\mu)t + \lambda_p |I| t) \le \exp(-(n-2|I|)(\lambda_p-\mu))\le 1$, which easily gives the desired estimate.
\end{proof}

Let $\mu_1, \mu_2$ be two probability measures on $\R$, and $\Lip(1)$ be the space of 1-Lipschitz functions $\R\mapsto [-1,1]$. We define
\begin{equation}
    m(\mu_1, \mu_2) := \sup_{g\in \Lip(1)} |\ddp{g}{\mu_1} - \ddp{g}{\mu_2}|. \label{eq:weakMetric}
\end{equation}
It is well known that $m$ is a distance metrizing the weak convergence (see e.g. \cite[Theorem 11.3.3]{Dudley:2002}). One easily checks that when $\mu_1, \mu_2$ correspond to two random variables $X_1,X_2$ on the same probability space then we have
\begin{equation}
    m(\mu_1, \mu_2) \leq \norm{X_1 - X_2}{1}{}\leq \sqrt{\norm{X_1 - X_2}{2}{}}. \label{eq:weakL2estimation}
\end{equation}

\begin{fact} \label{fact:tmp19_crit} Let $\cbr{X_t}_{t\geq 0}$ be the OU branching system starting from $x$ and $\lambda_p<2\mu$. For any $n \geq 2 $ there exists a function $l_n:\R_+ \mapsto \R_+$, fulfilling $\lim_{s\rightarrow 0} l_n(s) = 0$ and such that for any canonical $f_1, f_2 \in \pspace{}(\R^{nd})$ and any $t > 1$ we have
    \[
        m(\mu_1,\mu_2) \leq l_n(\norm{ f_1( 2n\cdot) - f_2(2n\cdot) }{\pspace{}}{}),
    \]
    where $\mu_1 \sim |X_t|^{-n/2} U_t^n(f_1)$, $\mu_2 \sim |X_t|^{-n/2} U_t^n(f_2)$ (the $U$-statistics are considered here conditionally on $Ext^c$).
\end{fact}

\begin{proof}
Let us fix $g \in \Lip(1)$. We consider
    \begin{multline*}
        L:=|\ddp{g}{\mu_1} - \ddp{g}{\mu_2}| = \left| \ev{} g\rbr{ |X_t|^{-n/2} U_t^n(f_1)} - \ev{} g\rbr{|X_t|^{-n/2} U_t^n(f_2)} \right| \\
        \leq  \ev{}\left| g\rbr{ |X_t|^{-n/2} U_t^n(f_1)} -  g\rbr{|X_t|^{-n/2} U_t^n(f_2)} \right|.
    \end{multline*}
    Let $h(x):=f_1(2nx)-f_2(2nx)$, take $r := \|h\|_{\pspace{}}^{1/n}$ and assume that $r < 1$. Then by Corollary \ref{cor:helper} we get
\begin{align*}
\ev{}\left| g\rbr{ |X_t|^{-n/2} U_t^n(f_1)} -  g\rbr{|X_t|^{-n/2} U_t^n(f_2)} \right|1_{\cbr{|X_t| \ge re^{-\lambda_pt}}} \le C_n\exp\cbr{c_n \norm{x}{}{}}\|h\|_{\pspace{}}^{1/2}.
\end{align*}
On the other hand,
\begin{align*}
\ev{}_x\left| g\rbr{ |X_t|^{-n/2} U_t^n(f_1)} -  g\rbr{|X_t|^{-n/2} U_t^n(f_2)} \right|1_{\cbr{|X_t| < re^{-\lambda_pt}}} \le 2\norm{g}{\infty}{}\pr{|X_t|e^{-\lambda_p t} < r}.
\end{align*}
Since on $Ext^c$ we have $|X_t|\geq 1$ and $|X_t|e^{-\lambda_p t}$ converges  to an absolutely continuous random variable,
\begin{displaymath}
\lim_{r\to 0+}\sup_t \pr{|X_t|e^{-\lambda_p t} < r} = 0,
\end{displaymath}
which ends the proof.
\end{proof}

An analogous proof using Corollary \ref{cor:helper_crit} gives a counterpart of the above fact in the critical case.
\begin{fact} \label{fact:tmp19} Let $\cbr{X_t}_{t\geq 0}$ be the OU branching system starting from $x$ and $\lambda_p=2\mu$. For any $n \geq 2 $ there exists a function $l_n:\R_+ \mapsto \R_+$, fulfilling $\lim_{s\rightarrow 0} l_n(s) = 0$ and such that for any canonical $f_1, f_2 \in \pspace{}(\R^{nd})$ and any $t > 0$ we have
    \[
        m(\mu_1,\mu_2) \leq l_n(\norm{ f_1( 2n\cdot) - f_2(2n\cdot) }{\pspace{}}{}),
    \]
    where $\mu_1 \sim (t|X_t|)^{-n/2} U_t^n(f_1)$, $\mu_2 \sim (t|X_t|)^{-n/2} U_t^n(f_2)$ (the $U$-statistics are considered here conditionally on $Ext^c$).
\end{fact}


\subsection{Proofs of main theorems}
\label{sub:proof_of_}
\subsubsection{The law of large numbers}

\begin{proof}[Proof of Theorem \ref{fact:multiLLD}]
    Consider the random probability measure $\mu_t = |X_t|^{-1} X_t$ (recall that formally we identify $X_t$ with the corresponding counting measure). By Theorem \ref{thm:LLN} with probability one (conditionally on $Ext^c$), $\mu_t$ converges weakly to $\eq$. Thus, by Theorem 3.2 in \cite{Billingsley:1999cl}, $\mu_t^{\otimes n}$ converges weakly to $\eq^{\otimes n}$. But
    $\langle f,\mu_t^{\otimes n}\rangle = |X_t|^{-n}V_t^n(f)$, which gives the almost sure convergence $|X_t|^{-n}V_t^n(f) \to \langle f,\varphi\rangle$. Now it is enough to note that the number of ''off-diagonal'' terms in the sum \eqref{eq:inex2} defining $U_t^n(f)$ is of order $|X_t|^{n-1}$ and use the fact that $|X_t| \to \infty$ a.s. on $Ext^c$.

The proof for $f\in \pspace{}(\R^{nd})$ follows directly from the central limit theorems from Section \ref{sec:last} (we will not use Theorem \ref{fact:multiLLD} in their proof). 
\end{proof}

\subsubsection{CLT -- slow branching rate}

\begin{proof}[Proof of Fact \ref{fact:well-posed}] \label{proof:well-posed} The sum \eqref{eq:limitVaribale} is finite hence it is enough to prove our claim for one $L(f, \gamma)$. Without loss of generality let us assume that $E_\gamma=\cbr{(1,2), (3,4), (2k-1,2k)} $ and $A_\gamma = \cbr{2k+1, \ldots, n}$ (we recall notation in Section \ref{sec:definitions_and_notation}). Using the same notation as in \eqref{eq:mu2} we write
    \begin{multline*}
            J(z_{2k+1}, \ldots, z_n) := \rbr{\prod_{(j,k)\in E_\gamma} {\int \mu_2({\dd{z_{j,k}}}) }} H(f)( u_1, u_2, \ldots, u_n )\\ =
                \int_{D} \int_{D}\ldots \int_{D} H(f)(z_1, z_1, z_2, z_2, \ldots, z_{k}, z_{k}, z_{2k+1} \dots, z_n )\mu_2\rbr{\dd{z_1}} \mu_2\rbr{\dd{z_2}} \ldots \mu_2\rbr{\dd{z_k}}.
    \end{multline*}
    where $D:=\R_+\times \Rd$. We know that $L(f, \gamma) = I_{n-2k}(J(x_{2k+1}, \ldots, x_n))$. By the properties of the multiple stochastic integral \cite[Theorem 7.26]{Janson:1997fk} we know that $\ev{L(f, \gamma)}^2 \cleq \rbr{\prod_{i\in \cbr{2k+1, \ldots, n}} \int_D \mu_1(dz_i)} |J(z_{2k+1}, \ldots, z_n)|^2 $. Therefore we need to estimate
    \begin{multline}
        \int_{D} \int_{D}\ldots \int_{D}  H(f)(z_1^1, z_1^1, z_2^1, z_2^1, \ldots, z_{k}^1, z_{k}^1, z_{2k+1} \dots, z_n ) H(f)(z_1^2, z_1^2, z_2^2, z_2^2, \ldots, z_{k}^2, z_{k}^2, z_{2k+1} \dots, z_n ) \\
        \mu_2\rbr{\dd{z_1^1}} \mu_2\rbr{\dd{z_2^1}} \ldots \mu_2\rbr{\dd{z_k^1}} \mu_2\rbr{\dd{z_1^2}} \mu_2\rbr{\dd{z_2^2}} \ldots \mu_2\rbr{\dd{z_k^2}} \mu_1\rbr{\dd{z_{2k+1}}}  \ldots \mu_1\rbr{\dd{z_{n}}}  \leq (*). \label{eq:integral}
    \end{multline}
    We will now estimate $H(f)(z_1, z_2, \ldots, z_n)$. Let $Y_1, Y_2, \ldots, Y_n$ be i.i.d., $Y_i \sim \eq$. We define $Y_i(t) = x_i e^{-\mu t} + ou(t) Y_i $. Recall our notation $z_i = (s_i,x_i)$ and let $I = \{i \in \{1,\ldots,n\} \colon s_i \ge 1\}$, $I^c = \{1,\ldots,n\}\setminus I$. Let $\hat{f} := \hat{f}_{I}$ be defined according to \eqref{eq:smoothing}.

    Using \eqref{eq:semi-group}, Lemma \ref{lem:derivatives}, the assumption that $f$ is canonical and the semigroup property, we can rewrite \eqref{eq:defH} as
    \begin{align*}
            H(f)(z_1, z_2, \ldots, z_n) &:= \ev{} f(Y_1(s_1), \ldots, Y_n(s_n)) = \ev{} \hat{f}((Y_i(s_i-1))_{i\in I}, (Y_i(s_i))_{i\in I^c})\\
&=\ev{}\:\: \sum_{(\epsilon_i)_{i\in I} \in \{0,1\}^I} (-1)^{\sum_{i\in I} \epsilon_i} \hat{f}\bigg(\Big(Y_i(s_i-1) + \epsilon_i ({Y}_i-{Y}_i(s_i-1))\Big)_{i\in I}, \Big(Y_i(s_i)\Big)_{i\in I^c}\bigg)\\
            &= \ev{}  \bigg(\prod_{i\in I}\int_{{Y}_i}^{{Y}_i(s_i-1)} \dd{y_i} \bigg) \frac{\partial^{|I|}}{\partial y_I} \hat{f}\bigg(\Big(y_i\Big)_{i\in I}, \Big(Y_i(s_i)\Big)_{i\in I^c} \bigg).
    \end{align*}
By Lemma \ref{lem:approximationQuality} we have (in order to simplify the notation we calculate for $d=1$, the general case is an easy modification)
    \begin{align*}
        &|H(f)(z_1, z_2, \ldots, z_n)| \cleq \norm{f}{\pspace{}}{} \ev{} \prod_{i\in I} \left|\int_{{Y}_i}^{{Y}_i(s_i-1)} e^{|y_i|}\dd{y_i}\right| \prod_{i\in I^c}e^{|Y_i(s_i)|}\\
        &\leq  \norm{f}{\pspace{}}{} \prod_{i\in I} \ev{} \left| \int_{{Y}_i}^{{Y}_i(s_i-1)} \rbr{e^{y_i} + e^{-y_i}}\dd{y_i}\right|\prod_{i\in I^c} \ev{} e^{|Y_i(s_i)|} \\
        &\leq \norm{f}{\pspace{}}{}\bigg( \prod_{i\in I} \Big(\ev{} |\exp(Y_i(s_i-1)) - \exp(Y_i)| + \ev{} |\exp(-Y_i(s_i-1)) - \exp(-Y_i)|\Big)\bigg)\ev{} \prod_{i\in I^c} e^{|Y_i(s_i)|}.
    \end{align*}
    Obviously for any $x,y \in \R$ we have $\max\rbr{\exp(x),\exp(y)}\leq \exp(x)+\exp(y)$. Therefore by the mean value theorem we get
    \[
        \ev{} |\exp(Y_i(s_i-1)) - \exp(Y_i)| \leq \ev{} |Y_i(s_i-1) - Y_i| \exp(Y_i(s_i-1)) + \ev{} |Y_i(s_i-1) - Y_i| \exp(Y_i).
    \]
    Using the Schwarz inequality and performing easy calculations we get
\begin{multline*}
    \ev{} |\exp(Y_i(s_i-1)) - \exp(Y_i)| \\\leq \sqrt{\ev{}(Y_i(s_i-1) - Y_i)^2}\rbr{\sqrt{\ev \exp(2Y_i(s_i-1))} + \sqrt{\ev \exp(2Y_i)}} \cleq \exp\rbr{2|x_i|} e^{-\mu s_i}.
\end{multline*}

Similarly
\begin{displaymath}
\ev{} |\exp(-Y_i(s_i-1)) - \exp(-Y_i)| \cleq \exp\rbr{2|x_i|} e^{-\mu s_i}.
\end{displaymath}
Since we also have $\ev{} \exp(|Y_i(s_i)|) \cleq \exp(|x_i|)$, we have thus proved that
\begin{align}\label{eq:nonempty_case}
        |H(f)(z_1, z_2, \ldots, z_n)| \cleq  \norm{f}{\pspace{}}{} e^{2\sum_{i=1}^n|x_i|} \prod_{i\in I}  e^{-\mu s_i}
\end{align}

We use the above inequality to estimate (\ref{eq:integral}). To this end let us denote $D_0 = [0,1)\times \R^d$, $D_1 = [1,\infty)\times \R^d$. To simplify the notation let us introduce the following convention. For subsets $I_1,I_2 \subset   \{1,\ldots,k\}$, $I_3 \subset \{2k+1,\ldots,n\}$ and $i \in \{1,\ldots,n\}$ we will write $I_j(i) = 1$ if $i \in I_j$ and $I_j(i) = 0$ otherwise. Let us also denote
\begin{displaymath}
F(\mathbf{z}) = H(f)(z_1^1, z_1^1, z_2^1, z_2^1, \ldots, z_{k}^1, z_{k}^1, z_{2k+1} \dots, z_n ) H(f)(z_1^2, z_1^2, z_2^2, z_2^2, \ldots, z_{k}^2, z_{k}^2, z_{2k+1} \dots, z_n )
\end{displaymath}
for $\mathbf{z} = (z_1^1,\ldots,z_k^1,z_1^2,\ldots,z_k^2,z_{2k+1},\ldots,z_n)$.

With this notation we can estimate \eqref{eq:integral} as follows

\begin{align*}
(*) &= \sum_{I_1,I_2 \subset\{1,\ldots, k\}}\sum_{I_3 \subset \{2k+1,\ldots,n\} } \bigg(\prod_{i=1}^k \int_{D_{I_1(i)}} \mu_2(\dd{z_i^1}) \bigg)\bigg(\prod_{i=1}^k \int_{D_{I_2(i)}} \mu_2(\dd{z_i^2}) \bigg)\bigg(\prod_{i=2k+1}^n\int_{D_{I_3(i)}} \mu_2(\dd{z_i})\bigg)F(\mathbf{z})\\
& =: \sum_{I_1,I_2 \subset\{1,\ldots, k\}}\sum_{I_3 \subset \{2k+1,\ldots,n\}} A(I_1,I_2,I_3).
\end{align*}

Now, using (\ref{eq:nonempty_case}) in combination with the Fubini theorem, the definition of the measures $\mu_i$ (given in Section \ref{sec:slowBranching}) and our assumption $\lambda_p < 2\mu$, we get

\begin{align*}
A(I_1,I_2,I_3) \cleq &\norm{f}{\pspace{}}{2} \bigg(\prod_{i=1}^k \int_{D_{I_1(i)}} e^{4|x_i^1| - 2I_1(i)\mu s_i^1} \mu_2(\dd{z_i^1}) \bigg)\bigg(\prod_{i=1}^k \int_{D_{I_2(i)}} e^{4|x_i^2| - 2I_2(i)\mu s_i^2}\mu_2(\dd{z_i^2}) \bigg)\times \\
&\times \bigg(\prod_{i=2k+1}^n\int_{D_{I_3(i)}} e^{4|x_i| - 2I_3(i)\mu s_i}\mu_1(\dd{z_i})\bigg) \cleq \|f\|_{\pspace{}}^2.
\end{align*}

To conclude the proof we use the fact that $f\mapsto L(f,\gamma)$ is linear and $\norm{\cdot}{\pspace{}}{}$ is a norm.
\end{proof}

\begin{proof}[Proof of Theorem \ref{thm:ustatistics-slow}] For simplicity we concentrate on the third coordinate. The joint convergence can be easily obtained by a modification of the arguments below (using the joint convergence in Theorem \ref{thm:clt1} for $n=1$). In the whole proof we work conditionally on the set of non-extinction $Ext^c$.

Let us consider bounded continuous functions $f_1^l, f_2^l, \ldots, f_n^l$, $l=1,\ldots,m$, which are centred with respect to $\eq$ and set ${f}_l:=\otimes_{i=1}^n{f_i^l} $ and $f = \sum_{l=1}^m f_l$.
	In this case the $U$-statistic \eqref{eq:ustat} writes as
	\[
		U^n_t(f) = \sum_{l=1}^m \sum_{\substack{i_1, i_2, \ldots, i_n = 1, \\i_j\neq i_k, \text{ for } j\neq k}}^{|X_t|} {f}_1^l(X_t(i_1)) {f}_2^l(X_t(i_2)) \ldots  {f}_n^l(X_t(i_n)).
	\]
 Let $\gamma$ be a Feynman diagram labeled by $\cbr{1,2,\ldots,n}$, with edges $E_\gamma$ and unpaired vertices $A_\gamma$. Let
		\[
			S(\gamma):=\sum_{\substack{i_1, i_2, \ldots, i_n = 1, \\i_j = i_k, \text{ if } (j,k) \in E_\gamma}}^{|X_t|} f (X_t(i_1), X_t(i_2), \ldots,  X_t(i_n)).
		\]
		Decomposition \eqref{eq:inex} writes here as
\begin{equation}
		 U^n_t(f) = \sum_\gamma (-1)^{r(\gamma)} S(\gamma) + R. \label{eq:decomposition2}
\end{equation}
where the sum spans over all Feynman diagrams labeled by $\cbr{1, 2, \ldots, n}$, and the remainder $R$ is the sum of $V$-statistics corresponding to partitions of $\{1,\ldots,n\}$ containing at least one set with more than two elements. We will prove that $|X_t|^{-(n/2)} R \rightarrow 0$. To this end let us consider  $\rbr{\bigcup_r A_r} \cup \rbr{\bigcup_r B_r} \cup \rbr{\bigcup_r C_r} = \cbr{1,2,\ldots, n}$, a partition in which $|A_r|\geq 3$, $|B_r|=2$ and $|C_r|=1$. Each set in this partition denotes which indices are the same (e.g. if $\cbr{1,2,3} = A_1$ then $i_1 = i_2 = i_3$ in our sum). By our assumption there is at least one set $A_r$. For any $l=1,\ldots,m$,  we write
\begin{multline*}
	K_t^l := \prod_r \rbr{|X_t|^{-|A_r|/2}\!\!\!\!\!\!\!\!\!\!\!\!\! \sum_{ \substack{i_{k_1} = i_{k_2}=\ldots = i_{k_m}\\\text{where} \cbr{k_1, k_2,\ldots,k_m}=A_r }}\!\!\!\!\!\!\!\! \!\!\!\!\!{f}_{k_1}^l(X_t(i_{k_1})) {f}_{k_2}^l(X_t(i_{k_2}))\ldots {f}_{k_m}^l(X_t(i_{k_m}))} \\ \prod_r \rbr{|X_t|^{-1}\!\!\!\!\!\!\!\!\!\!\!\!\! \sum_{ \substack{i_{k_1} = i_{k_2}\\\text{where} \cbr{k_1, k_2}=B_r }}\!\!\!\!\!\!\!\! \!\!\!\!\!{f}_{k_1}^l(X_t(i_{k_1})) {f}_{k_2}^l(X_t(i_{k_2}))} \prod_r \rbr{|X_t|^{-1/2}\!\!\!\!\!\!\!\sum_{ \substack{i_{k_1}\\\text{where} \cbr{k_1}=C_r }}\!\!\!\! \!\!\!\!\!{f}_{k_1}^l(X_t(i_{k_1})) }.
\end{multline*}

By Theorem \ref{thm:LLN} the first term converges to $0$ and the second one converges to a finite limit. The third term, by Theorem \ref{thm:clt1}, converges (in law) to the product of Gaussian random variables. We conclude that $K_t^l \rightarrow^d 0$. Thus only the first summand of \eqref{eq:decomposition2} is relevant for the asymptotics of $|X_t|^{-(n/2)}U_t^n(f)$. Consider now
\[
	|X_t|^{-(n/2)} S(\gamma) = \sum_{l=1}^m \prod_{(j,k)\in E_\gamma} \rbr{|X_t|^{-1}\sum_{i=1}^{|X_t|}{f}_{j}^l(X_t(i)){f}_{k}^l(X_t(i))} \prod_{r\in A_\gamma}\rbr{|X_t|^{-1/2}\sum_{i=1}^{|X_t|}{f}_{r}^l(X_t(i))}.
\]
Let us denote  $Z_{f_j^l}(t):=|X_t|^{-1/2}\sum_{i=1}^{|X_t|}{f}_{j}^l(X_t(i))$. By  Theorem \ref{thm:clt1} and the Cramér-Wold device we get that
\[
	(Z_{f_j^l}(t))_{1\le j \le n, 1\le l\le m} \rightarrow^d (G_{f_j^l})_{1\le j \le n, 1\le l\le m},
\]
where $(G_{f_j^l})_{1\le j \le n, 1\le l\le m}$ is a centred Gaussian vector with the covariances
\begin{equation}
	\cov(G_{f_j^l},G_{f_k^l}) = \ddp{\eq}{{f}_j^l {f}_k^l } + 2\lambda p  \inti  \ddp{\eq}{ \rbr{ e^{(\lambda_p/2) s}\T{s} {f}_j^l } \rbr{ e^{(\lambda_p/2) s}\T{s} {f}_k^l } }  \dd{s}. \label{eq:tmp17}
\end{equation}
On the other hand by Theorem \ref{thm:LLN} one could easily obtain
\[
	|X_t|^{-1}\sum_{i=1}^{|X_t|}{f}_{j}^l(X_t(i)){f}_{k}^l(X_t(i)) \rightarrow \ddp{{f}_{j}^l{f}_{k}^l }{\eq}, \:\:a.s.
\]
We conclude that
\[
	\cbr{|X_t|^{-(n/2)} S(\gamma) }_\gamma\rightarrow^d  \cbr{\sum_{l=1}^m \prod_{(j,k)\in E_\gamma} \ddp{{f}_{j}^l{f}_{k}^l }{\eq} \prod_{r\in A_\gamma} G_{f_r^l}}_\gamma =:\cbr{v(\gamma)}_\gamma.
\]

By decomposition \eqref{eq:decomposition2} and the considerations above we obtain
\[
	|X_t|^{-(n/2)t} U^n_t(f) \rightarrow^d \sum_\gamma (-1)^{r(\gamma)} v(\gamma) =: L.
\]
We will now show that $L$ is equal to $L_1(f)$ given by \eqref{eq:limitVaribale}. By linearity of $L_1(f)$ it is enough to consider the case of $m=1$. We will therefore drop the superscript and write $f_i$ instead of $f_i^l$. We denote by $P(f_i, f_j):= \int_D H(f_i \otimes f_j)(z,z) \mu_2(\dd{z})$. By \eqref{eq:tmp17} and definition of $\mu_2$ given in \eqref{eq:mu2}
\[
	L=\sum_\gamma (-1)^{r(\gamma)} \prod_{(j,k)\in E_\gamma} \ddp{{f}_{j}{f}_{k} }{\eq} \prod_{r\in A_\gamma} G_{f_r} = \sum_\gamma (-1)^{r(\gamma)} \prod_{(j,k)\in E_\gamma} (\ev{} G_{f_j} G_{f_k} - P(f_j,f_k) ) \prod_{r\in A_\gamma} G_{f_r}.
\]
We now adopt the notation that $\eta \subset \gamma$ when $E_\eta \subset E_\gamma$ and write
\begin{multline*}
	L= \sum_\gamma (-1)^{r(\gamma)} \sum_{\eta \subset \gamma} (-1)^{r(\eta)} \prod_{(j,k)\in E_\gamma \setminus E_\eta} \ev{} G_{f_j} G_{f_k} \prod_{(j,k)\in E_\eta}P(f_j,f_k)  \prod_{r\in A_\gamma} G_{f_r} \\
	= \sum_{\eta } \prod_{(j,k)\in E_\eta}P(f_j,f_k)  \sum_{\gamma \supset \eta} (-1)^{r(\gamma)-r(\eta)}  \prod_{(j,k)\in E_\gamma \setminus E_\eta} \ev{} G_{f_j} G_{f_k}  \prod_{r\in A_\gamma} G_{f_r}.
\end{multline*}
Let us notice that the inner sum can be written as
\[
\sum_{\sigma} (-1)^{r(\sigma)}  \prod_{(j,k)\in E_\sigma } \ev{} G_{f_j} G_{f_k}  \prod_{r\in A_\sigma} G_{f_r},
\]
where $\sigma$ runs over all Feynman diagrams on the set of vertices $A_\eta$. Thus by \cite[Theorem 3.4 and Theorem 7.26]{Janson:1997fk}  this equals $ I_{|A_\eta|}\rbr{H(\otimes_{i\in A_\eta} f_i )}$ and in consequence
\[
	L= \sum_{\eta }  \prod_{(j,k)\in E_\eta}P(f_j,f_k)  I_{|A_\eta|}\rbr{H(\otimes_{i\in A_\eta} f_i )}.
\]

It is easy to notice that in the case of $f = \otimes_{i=1}^n f_i$ the expression above is equivalent to \eqref{eq:limitVaribale}. 
Let us now consider a function $f \in \pspace{}$. We put $h(x) := f(2n x)$. By Lemma \ref{lem:approx2} we may find a sequence of functions $\cbr{h_k}_k \subset span(A)$ such that $h_k \rightarrow h$ in $\pspace{}$. Next we define $f_k(x) := h_k(x/2n)$. Now by Fact \ref{fact:tmp19} we may approximate $|X_t|^{-(n/2)} U^n_t(f)$ with $|X_t|^{-(n/2)} U^n_t(f_k)$ uniformly in $t$. This together with Fact \ref{fact:well-posed} and standard metric-theoretic considerations concludes the proof.
\end{proof}
\subsubsection{CLT -- critical branching rate}

\begin{proof}[Proof of Fact \ref{fact:well-posed-critical} and Theorem \ref{thm:ustatistics-critical}] As in the subcritical case, we will focus on the third coordinate.

Recall the notation of Lemma \ref{lem:approx2} and note  that $A \subset \mathcal{L}^{\otimes n}$, where $\mathcal{L}^{\otimes n}$ denotes the algebraic tensor product (or more precisely its standard realization as a subspace of $L_2(\R^{nd},\prod_{i=1}^n\Phi(dx_i))$). Thus the operator $L_2$ is clearly well defined on $A\cap Can$.
In the course of the proof of the limit theorem we will show that $L_2$ is bounded on this space equipped with the norm $\|\cdot\|_{\pspace{}}$ and thus by Lemma \ref{lem:approx2} it extends uniquely to a bounded operator $L_2 \colon Can \to L_2(\Omega,\mathcal{F},\p)$.

The proof is slightly easier than the one of Theorem \ref{thm:ustatistics-slow} as, because of larger normalization, the notion of $U$-statistics and $V$-statistics coincide in the limit. Let us consider bounded continuous functions $f_1^l, f_2^l, \ldots, f_n^l,$ $l = 1,\ldots,m$, which are centred with respect to $\eq$ and denote $f:=\sum_{l=1}^m \otimes_{i=1}^n f_i^l $. By \eqref{eq:decomposition2} we have
\[
    U^n_t(f) - V^n_t(f) = \: \:\sum_{\gamma \text{ with at least one edge}}\: \:(-1)^{r(\gamma)}S(\gamma) + R,
\]
simply by the fact that the Feynman diagram without edges corresponds to $V^n_t(f)$. Analogously as in the proof of Theorem \ref{thm:ustatistics-slow} we have $R\rightarrow 0$. Let us now fix some diagram $\gamma$. Without loss of generality we assume that $E_\gamma = \cbr{(1,2),(2,3),\ldots,(2k-1,2k)}$ for $k\geq 1$. We have
\[
    (t|X_t|)^{-n/2} S(\gamma) = \sum_{l=1}^m \prod_{i\leq k}\Big((|X_t|t)^{-1} \ddp{X_t}{f_{2i}^l f_{2i+1}^l}\Big) \prod_{i>2k}\Big((|X_t|t)^{-1/2} \ddp{X_t}{f_i^l}\Big).
\]

By Theorem \ref{thm:cltCritical} each of the factors for $i > k$ converges in distribution, whereas by Theorem \ref{fact:multiLLD} each factor for $i \le k$ converges in probability to $0$, in consequence so does $(t|X_t|)^{-n/2} S(\gamma)$, which shows that $ (t|X_t|)^{-n/2}U^n_t(f)$ and $(t|X_t|)^{-n/2}V^n_t(f)$ are asymptotically equivalent.

Let us denote  $Z_{f_j^l}(t):=(t|X_t|)^{-1/2}\sum_{i=1}^{|X_t|}{f}_{j}^l(X_t(i))$. By  Theorem \ref{thm:cltCritical} and the Cramér-Wold device we get that
\[
   \Big(e^{-\lambda_p t} |X_t|, (Z_{f_j^l}(t))_{1\le i\le n,l\le m} \Big) \rightarrow^d \Big(W,(G_{f_j^l})_{1\le i\le n,l\le m}\Big),
\]
where $W$ has exponential distribution with parameter $(2p-1)/p$, $(G_{f_i^l})_{i\le n,l\le m}$ is centered Gaussian with the covariances given by \eqref{eq:criticalCovariances}, moreover $W$ is independent of $(G_{f_j^l})_{1\le j\le n,l\le m}$. Thus
$$(e^{-\lambda_p t} |X_t|, (t|X_t|)^{-n/2}V_t^n(f) ) \rightarrow^d \rbr{W,\sum_{l=1}^m \prod_{j=1}^n G_{f_j^l}} = (W,L_2(f)).$$

In particular $(te^{\lambda_p t})^{-n/2}V_t^n(f) \rightarrow^d W^{n/2}L_2(f)$. Moreover, by Fact \ref{fact:approximation_critical},
$\|(te^{\lambda_p t})^{-n/2}V_t^n(f)\|_2 \le C\|f\|_{\pspace{}}$ for some constant $C$ and since this bound is independent of $t$  we get
$\|W^{n/2}L_2(f)\|_1 \le C\|f\|_{\pspace{}}$.
Recall that for any $n$ there exists a constant $c > 0$ such that for any $N > 0$ and any $N$-variate polynomial $P(G_1,\ldots,G_N)$ of degree $n$ in a Gaussian random vector $(G_1,\ldots,G_N)$, we have
\begin{equation}
	\p(|P(G_1,\ldots,G_N)| \ge c\|P(G_1,\ldots,G_N)\|_2) \ge c. \label{eq:tmp21}
\end{equation}
This is a well known fact in the theory of polynomial chaos and follows e.g. from a combination of Theorem 3.2.5, Theorem 3.2.10  and Proposition 3.3.1. (i.e. the Paley-Zygmund inequality) in \cite{Pena:1999bh}.  Note that $c$ does not depend on $N$, which is crucial in our application.
Let $f\in A\cap Can$. Now, for sufficiently large $D$, by Markov's inequality and the fact that $W$ is with probability one strictly positive,
\begin{displaymath}
\p(|L_2(f)| \ge D^2\|f\|_{\pspace{}}) \le \p(|W^{n/2}L_2(f)| \ge D\|f\|_{\pspace{}}) + \p(W^{-n/2} > D) \le \frac{C}{D} + \p(W^{-n/2} > D)
< c,
\end{displaymath}
which, together with \eqref{eq:tmp21}, implies that $\|L_2(f)\|_2 \le c^{-1}D^2\|f\|_{\pspace{}}$ and in consequence the operator $L_2$ is bounded on $A\cap Can$. In particular it admits a unique extension to a bounded operator $L_2\colon Can \to L_2(\Omega,\mathcal{F},\p)$, which proves Fact \ref{fact:well-posed-critical}.

Theorem \ref{thm:ustatistics-critical} follows now by the already established case of $f \in span(A)$, Fact \ref{fact:tmp19_crit} and standard approximation arguments.

\end{proof}

\subsubsection{CLT -- supercritical branching rate}

\begin{proof}[Proof of Theorem \ref{thm:ustatistics-fast}] Again we concentrate on the third coordinate. The joint convergence can be easily obtained by a modification of the arguments below (using the joint convergence in Theorem \ref{thm:clt2} for $n=1$). 

First, note that $U$-statistics and $V$-statistics are asymptotically equivalent. The argument is analogous to the one presented in the proof of Theorem \ref{thm:ustatistics-critical}, since under assumption $\lambda_p>2\mu$ we have
\begin{displaymath}
\frac{|X_t|}{\exp(2(\lambda_p - \mu)t)} \to 0\; \textrm{a.s.}
\end{displaymath}
as $n \to \infty$ and consequently we can disregard the sum over all multi-indices $(i_1,\ldots,i_n)$ in which the coordinates are not pairwise distinct.

 Let us consider bounded continuous functions $f_1^l, f_2^l, \ldots, f_n^l \colon \R^d \to \R,$ $l = 1,\ldots,m$, which are centred with respect to $\eq$ and denote $f:=\sum_{l=1}^m \otimes_{i=1}^n f_i^l $. By Theorem \ref{thm:clt2} for $n=1$ we have
\[
    e^{-n(\lambda - \mu)t} V_t^n(f) = \sum_{l=1}^m \prod_{i=1}^n\rbr{e^{-n(\lambda_p - \mu)t} V_t^1(f_i^l)} \rightarrow  \sum_{l=1}^m\prod_{i=1}^n \rbr{\ddp{\grad f_i^l}{\eq}\circ H_\infty } = {L}_3(f),\quad \text{in probability}.
\]
Before our final step we recall that the convergence in probability can be metrised by $d(X,Y) := \ev{}\rbr{\frac{|X-Y|^2}{|X-Y|^2+1}}\leq \ev{} |X-Y|^2$. Let us now consider a function $f \in \pspace{}$. By Lemma \ref{lem:approx2} we may find a sequence of functions $\cbr{f_k} \subset span(A)$ such that $f_k \rightarrow f$ in $\pspace{}$. Now by Corollary \ref{cor:helper_supercrit}  we may approximate $e^{-n(\lambda_p-\mu)t} U^n_t(f)$ with $e^{-n(\lambda_p-\mu)t} U^n_t(f_k)$ uniformly in $t$ in the sense of metric $d$. Moreover, one can easily show that $\lim_{k\rightarrow +\infty}d(\tilde{L}_3(f_k), \tilde{L}_3(f))= 0$. This concludes the proof.
\end{proof}

\section{Remarks on the non-degenerate case\label{sec:last}}

Let us remark that as in the case of $U$-statistics of i.i.d. random variables, by combining the results for completely degenerate $U$-statistics with the Hoeffding decomposition, we can obtain limit theorems for general $U$-statistics, with normalization, which depends on the order of degeneracy of the kernel. For instance, in the slow branching case Theorem \ref{thm:ustatistics-slow}, the Hoeffding decomposition and the fact that $\Pi_k \colon \pspace{}(\R^{nd}) \to \pspace{}(\R^{nk})$ is continuous, give the following
\begin{cor} Let $\lambda_p <2\mu$ and $f \in \pspace{}(\R^{nd})$ be symmetric and degenerate of order $k-1$. Then conditionally on $Ext^c$, $|X_t|^{-(n-k/2)}U_t^n(f - \langle f, \eq^{\otimes n}\rangle)$ converges in distribution to $\binom{n}{k}L_1(\Pi_k f)$.
\end{cor}

Similar results can be derived in the remaining two cases. Using the fact that on the set of non-extinction $|X_t|$ grows exponentially in $t$, we obtain
\begin{cor}
Let $\lambda_p =2\mu$ and $f \in \pspace{}(\R^{nd})$ be symmetric and degenerate of order $k-1$. Then conditionally on $Ext^c$, $t^{-k/2}|X_t|^{-(n-k/2)}U_t^n(f - \langle f, \eq^{\otimes n}\rangle)$ converges in distribution to $\binom{n}{k}L_2(\Pi_k f)$.
\end{cor}

\begin{cor}
Let $\lambda_p >2\mu$ and $f \in \pspace{}(\R^{nd})$ be symmetric and degenerate of order $k-1$. Then conditionally on $Ext^c$, $\exp(-(\lambda_pn - \mu k)t)U_t^n(f - \langle f, \eq^{\otimes n}\rangle)$ converges in probability to $\binom{n}{k}W^{n-k}L_3(\Pi_k f)$.
\end{cor}

\bibliographystyle{abbrv}
\bibliography{branching}

\end{document}